\documentclass[aos]{imsart}
\usepackage{setup}
\usepackage{float}
\usepackage{newcommands}
\usepackage{constants}
\newconstantfamily{c}{symbol=c}
\begin{document}
	\begin{frontmatter}
	\title{Spatially Adaptive Online Prediction of Piecewise Regular Functions}              
	\runtitle{Online Prediction of Piecewise Regular Signals}  
	\runtitle{Online Prediction of Piecewise Regular Signals\;\;\;\;\;\;\;}

	\begin{aug}
		\author{\fnms{Sabyasachi} \snm{Chatterjee}\ead[label=e1]{sc1706@illinois.edu}}\footnote{{Supported by NSF Grant DMS-1916375}}
		\and
		\author{\fnms{Subhajit} \snm{Goswami}\ead[label=e2]{goswami@math.tifr.res.in}}\footnote{{Supported by the SERB grant SRG/2021/000032 and a grant from the Infosys Foundation }}
		\blfootnote{Author names are sorted alphabetically.}
		
		\runauthor{Chatterjee, S. and Goswami S.}
		
		\affiliation{University of Illinois at Urbana-Champaign and Tata Institute of Fundamental Research 
		}

		\address{117 Illini Hall\\
			Champaign, IL 61820 \\
			\href{mailto:sc1706@illinois.edu}{sc1706@illinois.edu}\\
			\phantom{as}}
		\address{1, Homi Bhabha Road\\
			Colaba, Mumbai 400005, India\\
			\href{mailto:goswami@math.tifr.res.in}{goswami@math.tifr.res.in}\\
			\phantom{ac}}
		
	\end{aug}

\begin{abstract}
	We consider the problem of estimating piecewise regular functions in an online setting, i.e., the data arrive 
	sequentially and at any round our task is to predict the value of the true function at the next revealed point 
	using the available data from past predictions. We propose a suitably modified version of a recently 
	developed online learning algorithm called the sleeping experts aggregation algorithm. We show that this 
	estimator satisfies oracle risk bounds simultaneously for all local regions of the domain. 
	As concrete instantiations of the expert aggregation algorithm proposed here, we study an online mean 
	aggregation and an online linear regression aggregation algorithm where experts correspond to the set of 
	dyadic subrectangles of the domain. The resulting algorithms are near linear time computable in the 
	sample size. We specifically focus on the performance of these online algorithms in the context of 
	estimating piecewise polynomial and bounded variation function classes in the fixed design setup. The 
	simultaneous oracle risk bounds we obtain for these estimators in this context provide new and improved 
	(in certain aspects) guarantees even in the batch setting and are not available for the state of the art batch 
	learning estimators. 
\end{abstract}

\end{frontmatter}

\section{Introduction}\label{sec:intro}

In this paper we revisit the classical problem of estimating piecewise regular functions from noisy 
evaluations. The theory discussed here is potentially useful for a general notion of piecewise regularity 
although we will specifically give attention to the problem of estimating 
piecewise constant or piecewise polynomial functions of a given degree $m \geq 0$ and bounded variation 
functions which are known to be well approximable by such piecewise constant/polynomial functions. A 
classical aim is to design adaptive estimators which adapt optimally to the (unknown) number of pieces of 
the underlying piecewise constant/polynomial function. For example, if the true signal is (exactly or close to) 
a piecewise constant function with unknown number of pieces $k$, then it 
is desirable that the estimator attains a near parametric $\tilde{O}(k/N)$ rate of convergence ($N$ is the 
sample size, $\tilde{O}(\cdot)$ indicates ``upto some {\em fixed} power of $\log N$''), for all possible values of $k.$ This desired notion is often shown by demonstrating that an 
estimator satisfies the so called oracle risk bound trading off a squared error approximation term and a 
complexity term.  Several nonparametric regression estimators, such as Wavelet Shrinkage 
\cite{donoho1994ideal},~\cite{donoho1998minimax}, Trend 
Filtering~\cite{mammen1997locally},~\cite{tibshirani2014adaptive},~\cite{tibshirani2020divided}, Dyadic 
CART~\cite{donoho1997cart}, Optimal Regression Tree~\cite{chatterjee2019adaptive}, are known to attain 
such an oracle risk bound in the context of estimating piecewise constant/polynomial functions. In this 
context, two natural questions arise which we address in this paper.

\begin{enumerate}
\item \textbf{Q1}: Consider the online version of the problem of estimating piecewise constant or piecewise polynomial functions, i.e.,	 the data arrive sequentially and at any round our task is to predict the value of the true function at the next revealed point using the available data from past predictions. \textit{Does there exist an estimator which attains an oracle risk bound (similar to what is known in the batch learning setting) in the online setting}? This seems a natural and perhaps an important question given applications of online learning to forecasting trends.

\medskip

\item  \textbf{Q2}: The oracle risk bounds available for batch learning estimators in the literature imply a 
notion of adaptivity of the estimator. This adaptivity can be thought of as a global notion of adaptivity as the 
risk bound is for the entire mean sum of squared error of the estimator. If it can be shown that an estimator 
satisfies an oracle risk bound locally, simultaneously over several subregions of the domain, then this will 
imply a local/spatial notion of adaptivity. We explain this more in Section~\ref{sec:adap}. This gives rise to 
our second question.~\textit{Does there exist an estimator which attains an oracle risk bound simultaneously 
	over several subregions of the domain in the online setting?}. Even in the batch learning setting, it is not 
known whether state of the art estimators such as Wavelet Shrinkage, Trend Filtering, Dyadic CART, Optimal 
Regression Tree satisfy such a simultaneous oracle risk bound. 
\end{enumerate}

The main purpose of this paper is to recognize, prove and point out that by using a suitably modified version of an online aggregation algorithm developed in the online learning community, it is possible to answer both the above questions in the affirmative. 


\subsection{\textbf{Problem Setting}}

Throguhout this paper, we will work with regression or signal denoising in the fixed lattice design setup 
where the underlying domain is  a $d$ dimensional grid $L_{d,n} = [n]^d \eqqcolon \{1, \ldots, n\}^d$. Here $n$ can 
be thought of as the sample size per dimension and the total sample size can be thought of as $N = n^d.$ All 
of what we do here is meaningful in the regime where $d$ is moderately low and fixed but $n$ is large. The 
specific dimensions of interest are $d = 1, 2$ or $3$ which are relevant for sequence, image or video denoising or forecasting respectively.

We will focus on the problem of noisy online signal prediction. Let $K$ denote the $d$-dimensional grid 
$L_{d, n}$ and we abbreviate $|K| = N$. Suppose $\theta^* \in \R^K$ is the true underlying 
signal and
\begin{equation}\label{eq:model}
y = \theta^* + \sigma \epsilon
\end{equation}
where $\sigma > 0$ is unknown and $\epsilon$ consists of independent, mean zero sub-Gaussian 
entries with unit dispersion factor.

Consider the following online prediction protocol. At round $t \in [1:N]$,
\begin{itemize}

\item An adversary reveals a  index $\rho(t) \in K$ such that $\rho(t) \notin \{\rho(1),\dots,\rho(t - 1)\}.$

\item Learner predicts $\hat{\theta}_{\rho(t)}$.

\item Adversary reveals $y_{\rho(t)} = \theta^*_{\rho(t)} + \epsilon_{\rho(t)}$, a noisy version of $\theta^*_{\rho(t)}.$
\end{itemize}

Note that $\rho$ turns out to be a possibly adversarially chosen permutation of the entries of $K$ (see the 
paragraph preceding Theorem~\ref{thm:generic_online}). At the end of $N$ rounds, the predictions of the learner are measured with the usual expected mean squared loss criterion given by 
$${\rm MSE}(\hat{\theta},\theta^*) \coloneqq \frac{1}{N} \, \E \|\hat{\theta} - \theta^*\|^2.$$

\begin{remark}
Clearly, this online setting is a harder problem than batch learning where we get to observe the whole array $y$ at once and we need to estimate $\theta^*$ by denoising $y.$ Therefore, any online learning algorithm can also be used in the batch learning setup as well.
\end{remark}

\subsection{\textbf{A Definition of Spatial Adaptivity}}\label{sec:adap}
In nonparametric function estimation, the notion of spatial/local adaptivity for an estimator is a highly desirable property. Intuitively, an estimator is  spatially/locally adaptive if it adapts to a notion of complexity of the underlying true function locally on every part of the spatial domain on which the true function is defined. It is accepted that wavelet shrinkage based estimators, trend filtering estimators or optimal decision trees are spatially adaptive in some sense or the other. However, the meaning of spatial adaptivity varies quite a bit in the literature. It seems that there is no universally agreed upon definition of spatial adaptivity. In this section, we put forward one way to give a precise definition of spatial adaptivity which is inspired from the literature on \textit{strongly adaptive online algorithms} (see, e.g., ~\cite{daniely2015strongly},~\cite{ adamskiy2012closer},~\cite{hazan2007adaptive}) developed in the online learning community. 
One of the goals of this paper is to convince the reader that with a fairly simple analysis of the proposed online learning algorithm it is possible to establish this notion of spatial adaptivity (precisely defined below) in a very general setting. 

Several batch learning estimators $\hat{\theta}$ satisfy the so called oracle risk bounds of the following form:
$${\rm MSE}(\hat{\theta},\theta^*) \leq \frac{1}{N} \inf_{\theta \in \R^N} \left(\|\theta - \theta^*\|^2 +\sigma^2 k_{{\rm comp}}(\theta) p(\log n)\right)$$
where $k_{{\rm comp}}$ denotes a complexity function defined on the vectors in $\R^m$ and $p(\log n)$ is some power of $\log n$. This is a risk bound which implies that the estimator $\hat{\theta}$ adapts to 
the complexity of the true signal $k_{{\rm comp}}(\theta).$

For example, optimal decision trees such as the Dyadic CART and the ORT estimator satisfy an oracle risk 
bound (see Section $7$ in~\cite{donoho1997cart} and Theorem $2.1$ in~\cite{chatterjee2019adaptive}),  so does the Trend Filtering estimator (Remark $3.1$ in~\cite{guntuboyina2020adaptive}) and more classically, the wavelet shrinkage based estimators (see Section $1.5$ in~\cite{donoho1994ideal}). Here, the complexity function $k_{{\rm comp}}$ is directly proportional to the number of constant/polynomial pieces for univariate functions. For multivariate functions, $k_{{\rm comp}}$ is still proportional to the number of constant/polynomial pieces; measured with respect to an appropriate class of rectangular partitions of the domain, see~\cite{chatterjee2019adaptive}. 


However, from our point of view, this type of oracle risk bound, while being highly desirable and guaranteeing adaptivity against  the complexity function $k_{{\rm comp}}$, is still a global adaptivity bound as the bound is for the mean squared error of the whole signal $\theta^*.$ A good notion of local/spatial adaptivity should reveal the adaptivity of the estimator to the local complexity of the underlying signal. This naturally motivates us to make the following definition of spatial adaptivity.

\textit{We say that an estimator $\hat{\theta}$ is spatially adaptive with respect to the complexity parameter $k_{{\rm comp}}$ and with respect to a class $\mathcal{S}$ of subregions or subsets of the domain $L_{d,n}$ if the following risk bound holds simultaneously for every $S \in \mathcal{S}$}, 
$${\rm MSE}(\hat{\theta}_{S},\theta^*_{S}) \leq \frac{1}{|S|} \inf_{\theta \in \R^{S}} \left(\|\theta - \theta^*_{S}\|^2 +\sigma^2 k_{{\rm comp}}(\theta) \:p(\log n)\right).$$

The above definition of spatial adaptivity makes sense because if the above holds simultaneously for every $S \in \mathcal{S}$, then the estimator $\hat{\theta}$ estimates at $\theta^*$ locally on $S$ with a rate of convergence that depends on the local complexity $k_{{\rm comp}}(\theta^*_{S}).$ We will prove that our proposed online learning estimator is spatially adaptive in the sense described above with respect to a large class of subregions $S.$

\subsection{\textbf{Summary of Our Results}}

\begin{enumerate}
\item We formulate a slightly modified version of the so called sleeping experts aggregation algorithm for a general class of experts and a general class of comparator signals. We then state and prove a general simultaneous oracle risk bound for the proposed online prediction algorithm; see Theorem~\ref{thm:generic_online}. This is the main result of this paper and is potentially applicable to several canonical estimation/prediction settings. 
\smallskip

\item We specifically study an online mean aggregation algorithm as a special instance of our general 
algorithm and show that it satisfies our notion of spatial adaptivity (see Theorem~\ref{thm:pcconst}) with 
respect to the complexity parameter that counts the size of the minimal rectangular partition of the domain 
$L_{d,n}$ on which the true signal $\theta^*$ is piecewise constant. Even in the easier offline setting, natural 
competitor estimators like Dyadic CART and ORT are not known to satisfy our notion of spatial adaptivity. 
Equipped with the spatially adaptive guarantee we proceed to demonstrate that this online mean aggregation 
algorithm  also attains spatially adaptive minimax rate optimal bounds (see 
Theorem~\ref{thm:TV_slow_rate}) for the bounded variation function class in general dimensions. This is 
achieved by combining Theorem~\ref{thm:pcconst} with known approximation theoretic results. Such 
a spatially adaptive guarantee as in Theorem~\ref{thm:TV_slow_rate} is not known to hold for the TV 
Denosing estimator, the canonical estimator used for estimating bounded variation functions.

\smallskip

\item We then study an online linear regression aggregation algorithm based on the Vovk-Azoury-Warmouth (VAW) forecaster (see~\cite{vovk1998competitive},~\cite{azoury2001relative}) as another instantiation of our 
general algorithm. We show that this algorithm satisfies our notion of spatial adaptivity (see 
Theorem~\ref{thm:pcpoly}) with respect to the complexity parameter which counts the size of the minimal 
rectangular partition of the domain $L_{d,n}$ on which the true signal $\theta^*$ is piecewise polynomial of 
any given fixed degree $ \geq 1$. Like in the case with piecewise constant signals discussed above, natural competitor estimators such as Trend Filtering or higher order versions of Dyadic CART are not known to satisfy our notion of spatial adaptivity even in the easier offline setting. We then demonstrate that this online 
linear regression aggregation algorithm also attains spatially adaptive minimax rate optimal bounds (see 
Theorem~\ref{thm:trendfilter_slow_rate}) for univariate higher order bounded variation functions. This is 
again achieved by combining Theorem~\ref{thm:pcpoly} with known approximation theoretic results. Such a 
spatially adaptive guarantee as in Theorem~\ref{thm:trendfilter_slow_rate} is not known to hold for the state 
of the art Trend Filtering estimator.
\end{enumerate}

\subsection{\textbf{Closely Related Works}}
In a series of recent papers~\cite{baby2019online},~\cite{baby2020adaptive},~\cite{baby2021optimal},~\cite{baby2021optimal2}, the authors there have studied online estimation of univariate bounded variation and piecewise polynomial signals. In particular, the paper~\cite{baby2021optimal} 
brought forward sleeping experts aggregation algorithms~\cite{daniely2015strongly} in the context of predicting univariate bounded variation functions. These works have been a source of inspiration for this current paper.

In a previous paper~\cite{chatterjee2019adaptive} of the current authors, offline estimation of piecewise 
polynomial and bounded variation functions were studied with a particular attention on obtaining adaptive 
oracle risk bounds. The estimators considered in that paper were optimal decision trees such as Dyadic 
CART~(\cite{donoho1997cart}) and related variants. After coming across the paper~\cite{baby2021optimal} 
we realized that by using sleeping experts aggregation algorithms, one can obtain oracle risk bounds 
in the online setting which would then be applicable to online estimation of piecewise polynomial and 
bounded variation functions in general dimensions. In this sense, this work focussing on the online problem is a natural follow up of 
our previous work in the offline setting.

The main point of difference of this work with the 	papers~\cite{baby2019online},~\cite{baby2020adaptive},~\cite{baby2021optimal} is that here we formulate a 
general oracle risk bound that works simultaneously over a collection of subsets of the underlying domain 
(see Theorem~\ref{thm:generic_online}).  We then show that this result can be used in conjunction with some approximation theoretic results (proved 
in~\cite{chatterjee2019adaptive}) to obtain 
spatially adaptive near optimal oracle risk bounds for piecewise constant/polynomial and bounded variation functions in general dimensions. To the best of our 
understanding, the papers ~\cite{baby2019online},~\cite{baby2020adaptive},~\cite{baby2021optimal} have 
not addressed function classes beyond the univariate case, nor do they address the oracle risk bounds for 
piecewise constant/polynomial functions. But {\em most importantly}, it appears that our work is the first, in 
the online setting, to formulate a simultaneous oracle risk bound as in Theorem~\ref{thm:generic_online} and realize that one can deduce from this near optimal risk bounds for several function classes of recent interest. 
We hope that Theorem~\ref{thm:generic_online} will find applications for several other function classes (see~Section~\ref{sec:otherfunctons} below).


\section{Aggregation of Experts Algorithm}\label{sec:algo}
In this section, we describe our main prediction algorithm. Our algorithm is a slight modification of the so called {\em Strongly Adaptive} online algorithms 
discussed in~\cite{hazan2007adaptive},~\cite{adamskiy2012closer},~\cite{daniely2015strongly}.

In this section $K$ could be any general finite domain like $L_{d,n}$. An \textit{expert} will stand for a set $S \subset K$ 
equipped with an online rule defined on $S$ where, by an {\em online rule} $r^{(S)}$ 
corresponding to $S$, we mean a collection of (measurable) maps $r^{(S)}_{U,s}: \R^{U} 
\rightarrow \R$ indexed by $U \subset S$ and $s \in S \setminus U$. 
Operationally, the expert corresponding to a subset $S \in \mathcal{S}$ containing $\rho(t)$ predicts at the 
revealed point $\rho(t)$ the number given by 
\begin{equation}\label{eq:expertrule}
	\hat{y}^{(S)}_{\rho(t)} = r^{(S)}_{\rho[1:(t - 1)] \cap S,\, \rho(t)} \left(y_{\rho[1:(t - 1)] \cap S}\right).
\end{equation} The display \eqref{eq:expertrule} defines a vector $\hat{y}^{(S)} \in \R^{S}$ containing the 
predictions of the expert corresponding to the subset $S$. A {\em family} of experts 
corresponds to a sub-collection $\mathcal{S}$ of subsets of $K$. For any choice of online rules for every $S 
\in \mathcal S$, we refer to them collectively as an online rule $\mb r$ associated to $\mathcal S$.

As per the protocol described in Section~\ref{sec:intro}, at the beginning of any round $t \in 
[N]$ where $|K| = N$, the data index $\rho(t)$ is revealed. At this point, the experts corresponding to 
subsets $S \in \mathcal{S}$ either not containing $\rho(t)$ or not having any data index revealed previously, 
become inactive. All other experts provide a prediction of their own. 

We are now ready to describe our aggregation algorithm $\mathcal A$ for a set of experts 
$\mathcal S.$ The input to this algorithm is the data vector $y \in \R^{K}$ which is revealed 
sequentially in the order given by the permutation $\rho$. We denote the output of this 
algorithm here by $\hat{y} \in \R^{K}$. At each round $t \in [N]$, the algorithm outputs a 
prediction $\hat{y}_{\rho(t)}$. Below and in the rest of the article, we use $T_{a}(x)$ to 
denote the truncation map $T_{a}(x) = \min\{\max\{x,-a\},a\}$.

\begin{tabular}{p{15cm}l}
	\hline
	{\bf Aggregation algorithm.} Parameters - subset of experts $\mathcal S$, online rule $\mathbf r = \{r^{(S)} : S \in \mathcal S\}$
	and truncation parameter $\lambda > 1$\\ 
	\hline
	\\
	Initialize $w_{S, 1} = \frac{1}{|\mathcal S|}$ for all $S \in \mathcal S$.\\
	For $t = 1, \ldots, N$:
	\begin{enumerate}
		\item Adversary reveals $\rho(t)$.
		
		\item  Choose a set of {\em active experts} $A_t$ as $$\{S \in \mathcal S: \rho(t) \in S\mbox{ and }\rho(t') \in S \mbox{ for some }t' < t\}$$
		if $t > 1$ and $\{S \in \mathcal S: \rho(t) \in S\}$ if $t = 1$.
		
		\item Predict $\hat{y}_{\rho(t)} = \sum_{S \in A_t}\hat w_{S, t}\, T_{\lambda}(\hat{y}^{(S)}_{\rho(t)})$ where $\hat w_{S, t} \coloneqq \frac{w_{S, t}}{\sum_{S\in A_t} w_{S, t}}$. 
		\item  Update $w_{S, t}$'s so that $w_{S, t+1} = w_{S, t}$ for $S\notin  A_t$ and 
		$$w_{S, t+1} = \frac{w_{S, t}\,\e^{-\alpha\, \ell_{S, t}}}{\sum_{S \in A_t}\hat w_{S, t}\, \e^{-\alpha\, \ell_{S, t}}} = \frac{w_{S, t}\,\e^{-\alpha\, \ell_{S, t}}}{\sum_{S \in A_t} w_{S, t}\, \e^{-\alpha\, \ell_{S, t}}} \, \sum_{S \in A_t} w_{S, t}$$
		otherwise, where $\ell_{S, t} \coloneqq \big(T_{\lambda}(y_{\rho(t)}) - T_{\lambda}(\hat{y}^{(S)}_{\rho(t)})\big)^2$ and $\alpha \coloneqq \frac{1}{8 \lambda^2}$.
	\end{enumerate}\\
	\hline
\end{tabular}

\begin{remark}
This algorithm is similar to the sleeping experts aggregation algorithm discussed 
in~\cite{daniely2015strongly} except that we apply this algorithm after truncating the data by the map 
$T_{\lambda}.$ Since, we are interested in (sub-Gaussian) unbounded  errors, our data vector $y$ need not 
be bounded which necessitates this modification. See Remark~\ref{remark:aggregrate_regret_trun1} below.
\end{remark}

In the sequel we will refer to our aggregation algorithm as $\mathcal A(\mathbf r, \mathcal S, \lambda)$. The 
following proposition guarantees that the performance of the above algorithm is not much worse as 
compared to the performance of any expert $S \in \mathcal{S}$; for any possible input data $y$.

\begin{proposition}[error comparison against individual experts for arbitrary data]\label{prop:aggregate_regret} 
For any ordering $\rho$ of $K$ and $S \in \mathcal S$, we have
	\begin{equation}\label{eq:aggregate_regret}
	\sum_{s \in S} \big(y_s - \hat{y}_s\big)^2\, \le \sum_{s \in S}
	\big(y_{s} - \hat y^{(S)}_{s} \big)^2 + 8 \lambda^2 \log \e|\mathcal S| + 2\|y_S - 
	\Pi_{\lambda} y_S\|^2 + 4 \lambda^2 \sum_{\substack{s \in S}} \mathrm{1}\left(|y_{s}| > \lambda\right),
\end{equation}
	where $\Pi_{\lambda} z$, for any vector $z \in \R^{A}$ with $|A| < \infty$, denotes the 
$\ell^2$-projection of $z$ onto the $\ell^\infty$-ball of radius $\lambda$, i.e., $(\Pi_{\lambda} z)_a = T_{\lambda}(z_a)$ for all $a \in A$.
\end{proposition}

\begin{remark}\label{remark:aggregrate_regret_trun0}
	A remarkable aspect of Proposition~\ref{prop:aggregate_regret} is that it holds for any input data $y \in 
	\R^{L_{d,n}}$. In particular, no probabilistic assumption is necessary. In the terminology of online learning, 
	this is said to be a prediction bound for individual sequences. Usually, such an individual sequence prediction 
	bound is stated for bounded data; see, e.g., \cite{hazan2007adaptive}, ~\cite{daniely2015strongly}. On the other hand, Proposition~\ref{prop:aggregate_regret} holds for any data because we have introduced a truncation parameter in our aggregation algorithm.
\end{remark}

\begin{remark}\label{remark:aggregrate_regret_trun1}
The effect of truncation is clearly reflected  in the last two terms of \eqref{eq:aggregate_regret}. The 
issue of unbounded data points in the \textbf{noisy} setup was dealt earlier in the literature --- see, e.g., 
\cite[Theorem~5]{2021optimal} --- by choosing a value of $\lambda$ so that \textbf{all} the 
datapoints lie within the interval $[-\lambda, \lambda]$ with some prescribed (high) probability $1 - \delta$. 
The comparison bounds analogous to \eqref{eq:aggregate_regret} (without the last two terms) then hold on this 
high probability event. However, the issue of how to choose $\lambda$ such that \textbf{all} the 
datapoints lie within the interval $[-\lambda, \lambda]$ is not trivial to resolve unless one knows something about the data generating mechanism. There is a simple way to get around this problem in the offline version by 
setting $\lambda = \max_{j \in [N]} |y_j|$ (see \cite[Remark~8]{2021optimal}) which is obviously \textbf{not} 
possible in the online setting.  Our version of the algorithm and the accompanying 
Proposition~\ref{prop:aggregate_regret} provide an explicit bound on the error due to truncation for arbitrary 
$y \in \R^K$ in the online problem. To the best of our knowledge, such a bound was not available in the literature in the current setup. An operational implication of Proposition~\ref{prop:aggregate_regret} is that 
even if a few data points exceed $\lambda$ in absolute value by not too great a margin, we still get effective 
risk bounds. 
\end{remark}


\begin{proof}
The proof is similar to the proof of regret bounds for exponentially weighted average forecasters with 
exp-concave loss functions (see, e.g., \cite{hazan2007adaptive}, \cite{cesa2006prediction}). However, we 
need to take some extra care in order to deal with our particular activation rule (see step~2) and obtain the 
error terms as in \eqref{eq:aggregate_regret}.

	Let us begin with the observation that the function $\e^{-\eta (x - z)^2}$, where 
	$\eta > 0$, is concave in $x$ for all $x, z \in [- 1/\sqrt{8\eta}, 1 / \sqrt{8\eta}]$. Clearly, 
	this condition is satisfied for $\eta = \alpha$ and all $x, z \in [-\lambda, \lambda]$. 
	Therefore, since $\hat y_{\rho(t)}$ is an average of $T_{\lambda}(\hat y^{(S)}_{\rho(t)})$'s (which by definition lie in $[-\lambda, \lambda]$) 
	with respective weights $\hat w_{S, t}$ (see step~$3$ in the algorithm), we can write using the Jensen's inequality,
	\begin{align*}
		\exp\left(-\alpha\left(\hat y_{\rho(t)} - T_{\lambda}(y_{\rho(t)})\right)^2\,\right) \ge \sum_{S\in A_t} \hat w_{S, t}\,\e^{-\alpha \ell_{S, t}}
	\end{align*}
	where we recall that $\ell_{S, t} \coloneqq \big(T_{\lambda}(y_{\rho(t)}) - T_{\lambda}(\hat y^{(S)}_{\rho(t)})\big)^2.$

	Fix a subset $S \in \mathcal S$ such that $S \in A_t$. Taking logarithm on both sides and using the 
	particular definition of updates in step~4 of 
	$\mathcal A$, we get
	\begin{equation*}
		\left(\hat y_{\rho(t)} - T_{\lambda}(y_{\rho(t)})\right)^2  - \ell_{S, t} \, \le  \, \alpha^{-1} \log \left( \frac{\e^{-\alpha \ell_{S, t}}}{\sum_{S' \in \mathcal S} \hat w_{S', t} \e^{-\alpha \ell_{S', 
					t}}}
				\right) = \alpha^{-1}\log \left(\frac{w_{S, t + 1}}{w_{S, t}}\right).
	\end{equation*}
However, since $w_{S, t + 1} = w_{S, t}$ and hence the logarithm is 0 whenever $S \not \in A_t$ (see step~4 
in the algorithm), we can add the previous bound over all $t$ such that $S \in A_t$ to deduce:
	\begin{equation*}
		\sum_{t: S  \in A_t} \left(\hat y_{\rho(t)} - T_{\lambda}(y_{\rho(t)})\right)^2 \le \sum_{t: S \in A_t} 
		\ell_{S, t} \, + \, \alpha^{-1} \sum_{t \in [N]} \log \left(\frac{w_{S, t + 1}}{w_{S, 
				t}}\right) \le \sum_{t: S \in A_t} \ell_{S, t} \, + \, \alpha^{-1} \log \e |\mathcal S|
	\end{equation*}
where in the final step we used the fact that $w_{S, 1} = \frac{1}{|\mathcal S|}$ and $w_{S, N+1} \le 1$. Now it 
follows from our activation rule in step~2 that $S \setminus \{t: S \in A_t\}$ is at most a singleton and hence
\begin{equation*}
\sum_{t: S \ni \rho(t)} \left(\hat y_{\rho(t)} - T_{\lambda}(y_{\rho(t)})\right)^2	\le \sum_{t: S \in A_t}  
\ell_{S, t} + 4 \lambda^2 + \alpha^{-1} \log \e |\mathcal S|
\end{equation*}
where we used the fact that both $\hat y_{\rho(t)}$ and $T_{\lambda}(y_{\rho(t)})$ lie in $[-\lambda, \lambda]$. We can now conclude the proof from the above display by plugging $x = y_{\rho(t)}$ and $z = 
\hat y_{\rho(t)}$ into 
	$$(x - z)^2 \le (T_{\lambda}(x) - z)^2 + 2 (T_{\lambda}(x) - x)^2 + 4 \lambda^2 \mathrm1({|x| > \lambda}), \,\, \mbox{ when } |z| \le \lambda$$ and also $x = y_s$ and $z = \hat y_s^{(S)}$ into $(T_{\lambda}(x) - T_{\lambda}(z))^2 \le (x - z)^2$ upon recalling  the fact that $|\hat y_{s}| \leq \lambda$ for all $s \in K$.
\end{proof}

\section{A General Simultaneous Oracle Risk Bound} \label{sec:main_result}
In this section we will state a general simultaneous oracle risk bound for online prediction of noisy signals. As in the last section, in this section also $K$ refers to any arbitrary finite domain. Recall from our setting laid out in the introduction that we observe a data vector $y \in \R^{K}$ in some order where we can write
$$y = \theta^*+ \sigma \ep$$ 
where $\sigma > 0$ is unknown and $\epsilon_t$'s are independent, mean zero sub-Gaussian variables with 
unit dispersion factor. For specificity, we assume in the rest of the paper that
\begin{equation*}
\max(\P[\ep_t \ge x ], 	\P[\ep_t \le -x ]) \le  2 \e^{-x^2 /2 },\: \:   \mbox{for all $x \ge 0$ and $t \in K$.}
\end{equation*}
Let us emphasize that the constant $2$ is arbitrary and changing the constant would only impact the 
absolute constants in our main result, i.e., Theorem~\ref{thm:generic_online} below . Our focus here is on 
estimating signals $\theta^*$ that are piecewise regular on certain sets as we now explain. Let $\mathcal S$ 
be a family of subsets of $K$ (cf. the family of experts $\mathcal S$ in our aggregation algorithm) and for 
each $S \in \mathcal S$, let $F_S \subset \R^S$ denote a class of functions defined on $S$.


We now define $\mathcal{P}$ to be the set of all partitions of $K$ all of whose constituent sets are elements of 
$\mathcal S$. For any partition $P \in \mathcal{P}$, define the class of signals 
\begin{equation}
\Theta_{P} = \Theta_{P}\big((F_S: S \in \mathcal S) \big) = \{\theta \in \R^K: \theta_{S} \in F_S \:\:\forall S \in P\}.
\end{equation}
In this section, when we mention a piecewise regular function, we mean a member of the set $\Theta_{P}$ for a partition $P \in \mathcal{P}$ with not too many constituent sets.

For example, in this paper we will be specifically analyzing the case when $K = L_{d,n} = [n]^d$ is 
the $d$ dimensional lattice or grid, $\mathcal S$ is the set of all (dyadic) rectangles of 
$L_{d,n}$ and $F_{S}$ is the set of polynomial functions of a given degree $m \geq 0$ on the 
rectangular domain $S$. Then, $\mathcal{P}$ becomes the set of all (dyadic) rectangular 
partitions of $L_{d,n}$ and $\Theta_{P}$ becomes the set of piecewise polynomial functions 
on the partition $P$.

Coming back to the general setting, to describe our main result, we need to define an additional quantity which is a property of the set of online rules $\{r^{(S)}: S \in \mathcal{S}\}.$ For any partition $P \in \mathcal{P}$ and any $\theta \in \Theta_{P}$, let  $\mathcal{R}(y, \theta, P) = \mathcal{R}(\mb r, y, 
\theta, P) > 0$ be defined as,
\begin{equation}\label{eq:regret_bnd} 
\mathcal{R}(y, \theta, P) = \sup_{\rho}\,\,\sum_{S \in P}\,\Big(\sum_{t: \rho(t) \in S}(y_{\rho(t)} - 
\hat{y}_{\rho(t)}^{(S)})^2  - \|y_S - \theta_{S} \|^2\Big)
\end{equation}
(recall the definition of $\hat{y}_{\rho(t)}^{(S)}$ from~\eqref{eq:expertrule}). Clearly $\mathcal{R}(y, \theta, P)$ is a random variable and we denote its 
expected value by $\overline{\mathcal{R}}(\theta,P)$.

Here is how we can interpret $\mathcal{R}(y, \theta, P).$ Given a partition $P \in \mathcal{P}$ consider the following prediction rule $r_{P}.$ For concreteness, let the partition $P = (S_1,\dots,S_k)$. At round $t$, there will be only one index $i \in [k]$ such that $\rho(t)$ is in $S_i$. Then the prediction rule $r_{P}$ predicts a value $\hat{y}^{(S_i)}_{\rho(t)}.$ In other words, $r_P$ uses the prediction of the expert corresponding to the subset $S_i$ in this round. Also, consider the prediction rule $r_{\theta}$ which at round $t$ predicts by $\theta_{\rho(t)}$ for any fixed vector $\theta \in \Theta_{P}.$ If we have the extra knowledge that the true signal $\theta^*$ indeed lies in $\Theta_{P}$ then it may be natural to use the above online learning rule $r_{P}$ if the experts are good at predicting signals (locally on the domain $S$) which lie in $F_{S}.$
\textit{We can now interpret $\mathcal{R}(y, \theta, P)$ as the excess squared loss or regret (when the array revealed sequentially is $y$) of the online rule $r_{P}$ as compared to the online rule $r_{\theta}.$}

In the sequel we use $\| \theta \|_{\infty}$ to denote the $\ell_\infty$-norm of the vector  $\theta$. We also 
extend the definitions of $\mathcal{P}$ and $\Theta_P$ (see around \eqref{def:ThetaP}) to include 
partitions of {\em subsets} of $K$ comprising only sets from $\mathcal S$. In particular, for any  $\mathsf T \subset K$, 
define $\mathcal{P}_{\mathsf T}$ to be the set of all partitions $P$ of $\mathsf T$ all of whose constituent 
sets are elements of $\mathcal S$. For any partition $P \in \mathcal{P}_{\mathsf T}$, define the class of 
signals 
\begin{equation}\label{def:ThetaP}
\Theta_{P} = \Theta_{P}\big((F_S: S \in \mathcal S, S \subset \mathsf T) \big) = \{\theta \in \R^{\mathsf T}: \theta_{S} \in F_S \:\:\forall S \in P\}.
\end{equation}

Let us now say a few words about the choice of ordering $\rho$ which we can generally think of as a 
stochastic process taking values in $K$. 
We call $\rho$ as {\em non-anticipating}  if, conditionally on $(\rho[1:t], y_{\rho[1:t-1]})$, $\ep_{\rho(t)}$ is distributed as $\ep_s$ on the event $\{\rho(t) = s\}$ for any $t \in [N]$. Such orderings include deterministic orderings and orderings that are independent of the data. But more generally, any ordering where $\rho(t)$ is allowed to depend on the data only through $y_{\rho[1: t-1]}$ is non-anticipating. 
In particular, an adversary can choose to reveal the next index after observing all of the past data and our actions.

We are now ready to state our general result. 
\begin{theorem}[General Simultaneous Oracle risk bound]\label{thm:generic_online}
	Let $K$ be a finite set. Fix a set of experts $\mathcal{S}$ equipped with online 
	learning rule $\textbf{r}.$ For each $S \in \mathcal S$, fix $F_S \subset \R^S$ to be a 
	class of functions defined on $S$. Suppose $y$ is generated from the 
	model~\eqref{eq:model} and is input to the algorithm 
	$\mathcal A(\mathbf r, \mathcal S, \lambda)$. Let us denote the output of the $\mathcal A(\mathbf r, 
	\mathcal S, \lambda)$ by $\hat{\theta}.$ Let $\mathsf T$ be any subset of $K.$ There exist absolute constants $c 
	\in (0, 1)$ and $C  > 1$ such that for any non-anticipating ordering $\rho$ of $K$, 
	\begin{align}\label{eq:generic_online}
		\E \|\hat y_{\mathsf T} - \theta_{\mathsf T}^* \|^2  \le \inf_{\substack{P \in {\mathcal P}_{\mathsf T},\\ \theta \in \Theta_P}} &\big( \|\theta_{\mathsf T}^* - \theta \|^2 + C \lambda^2 \,|P|\log 
		\e |\mathcal S| + \overline{\mathcal{R}}(\theta,P) \big) + C \lambda^2|\{s \in \mathsf T: 
		|\theta_s| > \lambda\}| \nonumber \\
		&
		+ C \|\theta_{\mathsf T}^* - \Pi_{\lambda}\theta_{\mathsf 
			T}^*\|^2 + C (\sigma^2 + \lambda^2) \sum_{s \in \mathsf T} \e^{-c \min\big(\frac{|\theta_s^* - \lambda|^2}{\sigma^2}, \frac{|\theta_s^{*} + \lambda|^2}{\sigma^2}\big)}.
	\end{align}
	In particular, there exists an absolute constant $C > 1$ such that for $\lambda \ge C (\sigma \sqrt{\log |\mathsf T|} \vee\|\theta^*\|_{\infty})$, one has
	\begin{align}\label{eq:generic_online2}
		{\rm MSE}(\hat \theta_{\mathsf T}, \theta_{\mathsf T}^*) \le \inf_{P \in {\mathcal P}_{\mathsf T},\, \theta \in 
			\Theta_P} \frac{1}{|\mathsf T|}&\big(\|\theta_{\mathsf T}^* - \theta \|^2 + C\, \lambda^2 \,|P|\log \e 	|\mathcal S| +
		\overline{\mathcal{R}}(\theta,P)\big) + 
		\frac{\sigma^2 + \lambda^2}{|\mathsf T|^2}\,.
\end{align}\end{theorem}

\begin{remark}\label{remark:truncation_parameter}
Our truncation threshold $C (\sigma \sqrt{\log |\mathsf T|} \vee\|\theta^*\|_{\infty})$ is comparable to the 
threshold  given in \cite[Theorem~5]{2021optimal} for Gaussian errors.
\end{remark}

We now explain various features and aspects of the above theorem.

\begin{itemize}
\item The reader can read the bound in~\eqref{eq:generic_online2} as
\begin{align*}
{\rm MSE}(\hat \theta_{\mathsf T}, \theta_{\mathsf T}^*) \le \inf_{\substack{P \in {\mathcal P}_{\mathsf T},\\ 
\theta \in \Theta_P}} \frac{1}{|\mathsf T|} \big(\underbrace{\|\theta_{\mathsf T}^* - \theta 
\|^2}_{T_1} + \underbrace{C\, \lambda^2 \,|P|\log \e |\mathcal S|}_{T_2} + 
\underbrace{\overline{\mathcal{R}}(\theta,P)}_{T_3}\big) + \textit{lower order term}.\end{align*}
Indeed, the only importance of the factor $\frac{1}{|\mathsf T|^{2}}$ in the last term of \eqref{eq:generic_online2} is that it is $o(\frac{1}{|\mathsf T|})$, i.e., of lower order 
than the principal term. Indeed, by suitably increasing the constant 
$C$, we can get any given power of $|\mathsf T|$ in the 
denominator.

\medskip

\item To understand and interpret the above bound, it helps to first consider $\mathsf T = K$ and 
then fix a partition $P \in \mathcal{P}_{K} = \mathcal{P}$ and a piecewise regular signal $\theta \in \Theta_{P}.$ We also keep in mind two prediction rules, the first one is the online rule $r_{P}$ and the second one is $r_{\theta}$ (both described before the statement of Theorem~\ref{thm:generic_online}).
The bound inside the infimum is a sum of three terms, $T_1,T_2$ and $T_3$ as in the last display. 
\begin{enumerate}

\item The first term $T_1$ is simply the squared distance between $\theta$ and $\theta^*.$ This term is obviously small or big depending on whether $\theta$ is close or far from $\theta^*.$

\item The second term captures the complexity of the partition $P$ where the complexity is simply the 
cardinality or the number of constituent sets/experts $|P|$ multiplied by log cardinality of the total number of 
experts $\log \e |\mathcal{S}|$. The reader can think of this term as the ideal risk bound achievable and 
anything better is not possible when the true signal $\theta^*$ is piecewise regular on $P.$ This term is 
small or big depending on whether $|P|$ is small or big.

\item The third term $T_3$ is $\mathcal{R}(\theta, P)$ which can be interpreted as the expected excess squared loss or regret of the online rule $r_{P}$ as compared to the prediction rule $r_{\theta}.$  This term is small or big depending on how good or bad is the online rule $r_{P}$ compared to the prediction rule $r_{\theta}.$
\end{enumerate}

\medskip

\item Our bound is an infimum over the sum of three terms $T_1,T_2,T_3$ for any $P \in \mathcal{P}$ and $\theta \in \Theta_{P}$ which is why we can think of this bound as an oracle risk bound in the following 
sense. Consider the case when $\mathsf T = K$ and $\theta^*$ lies in $\Theta_{P^*}$ for some $P^*$ which 
is of course unknown. In this case, an oracle who knows the true partition $P^*$ might naturally trust experts 
locally and use the online prediction rule $r_{P^*}.$ In this case, our bound reduces (by setting $P = P^*, \theta = \theta^*$) to the MSE incurred by this oracle prediction rule plus the ideal risk  $|P^*| \log \e 
|\mathcal{S}|$ term which is unavoidable. Because of the term $T_1$, this argument holds even if $\theta^*$ does not exactly lie in $\Theta_{P^*}$ but is very close to it. To summarize, our MSE bound ensures that we 
nearly perform as well as an oracle prediction rule which knows the \textit{true} partition corresponding to 
the target signal $\theta^*.$

\medskip

\item The term $T_{3}$ in the MSE bound in~\eqref{eq:generic_online2} behooves us to find experts with 
good online prediction rules. If each expert $S \in \mathcal{S}$ indeed is equipped with a good prediction 
rule, then under the assumption that $\theta^*$ is exactly (or is close to) piecewise regular on a partition 
$P^* \in \mathcal{P}$, the term $T_3 = \overline{\mathcal{R}}(\theta^*, P^*)$ will be small and our bound 
will thus be better. This is what we do in our example applications, where we use provably good online rules 
such as running mean or the online linear regression forecaster of Vovk~\cite{vovk1998competitive}. Infact, in each of 
the examples that we discuss subsequently in this paper, we bound the term $\overline{\mathcal{R}}(\theta, 
P)$ in two stages. First, we write
$$\E \mathcal{R}(y, \theta, P) \leq |P|\:\: \E \sup_{\rho, S \in P} \Big(\sum_{t: \rho(t) \in S}(y_{\rho(t)} - \hat{y}_{\rho(t)}^{(S)})^2  - \|y_S - \theta_{S} \|^2\Big).$$
Then in the second stage we 
obtain a bound on the expectation in the right side above by a log factor. Thus there is no real harm if the reader thinks of $\overline{\mathcal{R}}(\theta, P)$ as also being of the order $|P|$ up to log factors which is the ideal and unavoidable risk as mentioned before.

\medskip 
\item A remarkable feature of Theorem~\ref{thm:generic_online} is that the MSE bound in~\eqref{eq:generic_online2} holds \textit{simultaneously} for all sets $\mathsf T \subset K.$ Therefore, the 
interpretation that our prediction rule performs nearly as well as an oracle prediction rule holds locally for 
every subset or subregion $\mathsf T$ of the domain $K.$ This fact makes our 
algorithm provably \textit{spatially adaptive} to the class of all subsets of $K$ with respect to the complexity parameter proportional to $|P|$ in the sense described in Section~\ref{sec:adap}. The implications of this will be further discussed when we analyze online prediction of specific function classes in the next two sections. 



\end{itemize}

\begin{proof}[Proof of Theorem~\ref{thm:generic_online}]
Since $y = \theta^* + \sigma \epsilon$, we can write for any $S \in \mathcal S$,
	\begin{align*}
		\|y_S - \hat y_S\|^2 = \|\hat y_S - \theta^*_S \|^2 + 2\sigma \langle \epsilon_S, \hat y_S\rangle + \sigma^2 \|\epsilon_S\|^2.
	\end{align*}
However, since $\rho$ is non-anticipating and $\hat y_{\rho(t)}$ is measurable relative to $(\rho[1:t], y_{\rho[1: (t-1)]})$ 
and $\ep_s$'s  have mean zero, it follows from the previous display that
	\begin{align}\label{eq:l2decomp}
		\E \|y_S - \hat y_S\|^2 = \E \|\hat y_S - \theta^*_S \|^2 +  \sigma^2 \E \|\epsilon_S\|^2.
	\end{align}	
	On the other hand, adding up the upper bound on $\|y_S - \hat y_S\|^2$ given by 
	Proposition~\ref{prop:aggregate_regret} over all $S \in P$ for some $P \in \mathcal P$ we get 
	\begin{align*}
		\|y_{\mathsf T} - \hat y_{\mathsf T}\|^2\, \le \, \, & \sum_{S \in P}\,\sum_{\substack{t: \rho(t) \in S}} \big(y_{\rho(t)} - \hat y_{\rho(t)}^{(S)} \big)^2 + 2\|y_{\mathsf T} - 
		\Pi_{\lambda} y_{\mathsf T}\|^2 + 8 \lambda^2 |P| \log \e |\mathcal S| \nonumber \\  & + 4 \lambda^2\sum_{s \in \mathsf T} 1\{|y_s| > \lambda\}.
	\end{align*}Now taking expectations on both sides and using the definition of $\overline 
	\lambda(\theta, P)$ from \eqref{eq:regret_bnd}, we can write
	\begin{align}\label{eq:y_haty_dist}
		\E \|y_{\mathsf T} - \hat y_{\mathsf T}\|^2 \le \, \, & \E\|y_{\mathsf T} - \theta \|^2 + \overline \lambda(\theta, P) + 8  \lambda^2 \,|P|\log \e |\mathcal S| + 2\E \|y_{\mathsf T} - \Pi_{\lambda}y_{\mathsf T}\|^2 \nonumber \\ 
		&+ 4 \lambda^2 \sum_{s \in \mathsf T} \P(|y_s| > \lambda).
	\end{align}
	Since $\ep_s$'s have mean zero, we get by expanding $\|y_{\mathsf T} - 
	\theta \|^2$,
	\begin{align*}
		\E\|y_{\mathsf T} - \theta \|^2 \le \|\theta_{\mathsf T}^* - \theta \|^2  + \sigma^2 \E\|\ep_{\mathsf T}\|^2.
	\end{align*}
	Plugging this bound into the right hand side of \eqref{eq:y_haty_dist}, we obtain
	\begin{align*}
		\E \|y_{\mathsf T} - \hat y_{\mathsf T} \|^2 \le  & \:\:\|\theta_{\mathsf T}^* - \theta \|^2 + \sigma^2 \E\|\ep_{\mathsf T}\|^2 + \overline \lambda(\theta, P) + 8 \lambda^2 \,|P|\log \e |\mathcal S| + 2\E\|y_{\mathsf T} - \Pi_{ \lambda}y_{\mathsf T}\|^2 \nonumber \\ & + 4  \lambda^2 \sum_{s \in \mathsf T} \P(|y_s| >  \lambda).
	\end{align*}
	Together with \eqref{eq:l2decomp}, this gives us
	\begin{align*}
		\E \|y_{\mathsf T} - \theta_{\mathsf T}^*\|^2 \le & \:\:\|\theta_{\mathsf T}^* - \theta \|^2 + \overline \lambda(\theta, P) + 8  \lambda^2 \,|P|\log \e |\mathcal S| + 2\E\|y_{\mathsf T} - \Pi_{ \lambda}y_{\mathsf T}\|^2 \nonumber \\ 
		& + 4  \lambda^2 \sum_{s \in \mathsf T} \P(|y_s| >  \lambda).
	\end{align*}
	Minimizing the right hand side in the above display over all $P \in \mathcal P$ and $\theta 
	\in \Theta_P$, we get 
	\begin{align}\label{eq:tail_bnd_decomp}
		\E \|y_{\mathsf T} - \theta_{\mathsf T}^*\|^2 \le & \:\: \inf_{P \in \mathcal P,\, \theta \in \Theta_P} \big(\|\theta_{\mathsf T}^* - \theta \|^2 + \overline \lambda(\theta, P) + 8  \lambda^2 \,|P|\log 
		|\mathcal S|\big) + 2\E\|y_{\mathsf T} - \Pi_{ \lambda}y_{\mathsf T}\|^2 \nonumber \\ 
		& + 4  \lambda^2 \sum_{s \in \mathsf T} \P(|y_s| >  \lambda).\end{align} It only remains to verify the 
	bounds on the error terms due to truncation. Since $|\hat y_s| \le  \lambda$ by the design of our 
	algorithm, we have \begin{equation}\label{eq:thetas_big} (\hat y_s - \theta_s^*)^2 \le 
	2(\theta_s^* - T_{ \lambda}(\theta_s^*))^2  + 8  \lambda^2 \end{equation} for all $s \in \mathsf T$. We 
	will apply this naive bound whenever $|\theta_s^*| >  \lambda$. So let us assume that $s$ is such 
	that $|\theta_s^*| \le  \lambda$. Let us start by writing
	\begin{align*}
		\E(y_s - T_{\lambda}(y_s))^2 = I_+ + I_-,
	\end{align*}
	where, with $x_+ \coloneqq \max(x, 0)$ and $x_- \coloneqq -\min(x, 0)$, \begin{align*}I_+ \coloneqq \E(y_s - \lambda)_+^2 = 2\sigma^2 \int_{x > \frac {\lambda}{\sigma}} \big (x - \frac {\lambda}{\sigma} \big)\,\P\big[\ep_s > x - \frac{\theta_s^*}{\sigma}\big] \, dx,\:\: \mbox{and}\nonumber \\ 
		I_- \coloneqq \E(y_s + \lambda)_-^2 = -2 \sigma^2 \int_{x < -\frac{\lambda}{\sigma}} \big(x 
		+ \frac{\lambda}{\sigma}\big)\,\P\big[\ep_s < x - \frac{\theta_s^*}{\sigma}\big] \, dx.
	\end{align*} In writing these expressions we used the standard fact that $$\E(X - a)_+^2 = 2 \int_{x > a} (x - a) \P[X > x] \, dx.$$ We deal with $I_+$ first. Since $\theta_s^* \le \lambda$ and $\ep_s$ has 
	sub-Gaussian decay around $0$ with unit dispersion factor, we 
	can bound $I_{+}$ as follows:\begin{align} \label{eq:I+estimate} I_{+} =  2\sigma^2 \int_{x > \frac {\lambda - \theta_s^*}{\sigma}} \big (x - \frac {\lambda - \theta_s^*}{\sigma} \big)\,\P\big[\ep_s > x \big] \, dx \le C\sigma^2  \,\e^{-\frac{|\theta_s^* - \lambda|^2}{2 
				\sigma^2}}\end{align} where $C$ is an absolute constant. Similarly we can deduce \begin{align} \label{eq:I-estimate} I_{-} \le C\sigma^2  \,\e^{-\frac{|\theta_s^* - \lambda|^2}{2 \sigma^2}}.\end{align}
	On the other hand, for any $|\theta_s^*| \le \lambda$, we have \begin{align}\label{eq:gaussian_tail}\P[|y_s| > \lambda] \le \P \big[ \ep_s > \frac{\lambda - \theta_s^*}{\sigma} \big] + \P\big[ \ep_s < \frac{- \lambda - \theta_s^*}{\sigma}\big] \le 2\e^{- \frac{|\theta_s^* - \lambda|^2}{2\sigma^2} \wedge \frac{|\theta_s^{*} + \lambda|^2}{\sigma^2}}.\end{align}
	Finally we plug the estimates \eqref{eq:I+estimate}, \eqref{eq:I-estimate} and 
	\eqref{eq:gaussian_tail} into \eqref{eq:tail_bnd_decomp} when $|\theta_s^*| \le \lambda$, and the 
	estimate \eqref{eq:thetas_big} and the trivial upper bound on probabilities when 
	$|\theta_s^*| > \lambda$ to get \eqref{eq:generic_online}. \eqref{eq:generic_online2} then 
	follows immediately from \eqref{eq:generic_online} by choosing $C$ large 
	enough.\end{proof}

In the next few sections we will introduce and discuss online prediction rules for several classes of functions. 
In each case we will apply Theorem~\ref{thm:generic_online} to derive the corresponding risk bounds. The 
proofs of all the results are given in the Appendix section.
\section{Online Mean Aggregation over Dyadic Rectangles (OMADRE)}\label{sec:omadre}
In this section, we will specifically study a particular instantiation of our general 
algorithm (laid out in Section~\ref{sec:algo}) which we tentatively call as the Online Mean Aggregation over Dyadic Rectangles estimator/predictor (OMADRE). Here, $K = L_{d,n}$ and the set of experts corresponds to the set of all 
\textit{dyadic rectangles} of $L_{d,n}$. Some precise definitions are given below. 

An {\em axis aligned rectangle} or 
simply a rectangle $R$ is a subset of $L_{d, n}$ which is a 
product of intervals, i.e., $R = \prod_{i=1}^d[a_i, b_i]$ for some $1 \le a_i \le b_i \le n$; 
$i \in [d]$. A sub-interval of $[1, n]$ is called {\em dyadic} if it is of the form 
$((a-1)2^s, a2^s]$ for some integers $0 \le s \le k$ and $1 \le a \le 2^{k-s}$ where we assume 
$n = 2^k$ for simplicity of exposition. \textit{We call a rectangle dyadic if it is a product of 
dyadic intervals}.

Now we take our experts to be the dyadic sub-rectangles of $L_{d, n}$, i.e., 
in the terminology of Section~\ref{sec:main_result}, $\mathcal S$ is the set of dyadic sub-rectangles of $L_{d, n}$. We set $F_S = {\rm span}(\{1\})$ ---- the space of 
all constant functions on $S$ --- for all $S \in \mathcal S$. We also let $\mathcal P_{dp}$ be the 
set of all {\em dyadic rectangular partitions} of $L_{d, n}$ where a (dyadic) rectangular 
partition $P$ is a partition of $L_{d, n}$ comprising only (respectively, dyadic) rectangles. 
Since there are at most $2n$ dyadic sub-intervals of $[n]$, we note that \begin{equation}\label{eq:no_of_expert} 
|\mathcal S| = (2n)^d = 2^d N.
\end{equation}
 Under this setting, for any partition $P$ of $L_{d,n}$ the set $\Theta_{P}$ refers to the set of all arrays $\theta \in \R^{L_{d,n}}$ such that $\theta$ is constant on each constituent set of $P.$

	Finally we come to the choice of our online rule $\mb r$. It is very natural to consider the online averaging rule ${\mb r}$ 
	defined as:
	\begin{equation}\label{def:online_average}
	r_{U, s}^{(S)}((y_{u} : u \in U)) = \overline y_U  \mbox{ for all $U \subset S \in \mathcal S$ and $s \in S \setminus U$}.
	\end{equation}
By convention, we set $\overline y_{\emptyset} = 0$.



\begin{lemma}[Computational complexity of {\rm OMADRE}]\label{lem:complexity_online_mean}
There exists an absolute constant $C > 0$ such that the computational complexity, i.e., the number of 
elementary operations involved in the computation of OMADRE is bounded by $C N (\log_2 2n)^d$.
\end{lemma}

\begin{remark}
	The above computational complexity is near linear in the sample size $N$ but exponential in the dimension $d.$ Therefore, the estimators we are considering here are very efficiently computable in low dimensions which are the main cases of interest here. 
	\end{remark}


The OMADRE, being an instance of our general algorithm will satisfy our simultaneous oracle risk bound in Theorem~\ref{thm:generic_online}. This oracle risk bound can then be used to derive risk bounds for several function classes of interest. 
We now discuss two function classes of interest for which the OMADRE performs near optimally. 

\subsection{\textbf{Result for Rectangular Piecewise Constant Functions in General Dimensions}}\label{sec:piecewise_constant}
Suppose $\theta^*$ is piecewise constant on some unknown rectangular partition $P^*$ of the domain $K = 
L_{d,n}$. For concreteness, let the partition $P^* = (R_1,\dots,R_k)$. 
	An \textit{oracle predictor} $\hat{\theta}_{({\rm oracle})}$ --- which {\em knows} the minimal rectangular partition $(R_1,\dots,R_k)$ of $\theta^*$ {\em exactly} ---can simply use the online averaging prediction 
	rule given in~\eqref{def:online_average} separately within each of the rectangles$(R_1,\dots,R_k).$ By a 
	basic result about online mean prediction, (see Lemma~\ref{lem:online_mean_regret1}), it can be shown 
	that the MSE of this oracle predictor is bounded by $O(\frac{k \|\theta^*\|_{\infty} \log n}{N}).$ In words,  
	the MSE of this oracle predictor scales (up to a log factor which is necessary) like the \textit{number of 
		constant pieces of $\theta^*$ divided by the sample size $N$ which is precisely the parametric rate of convergence.}

A natural question is whether there exists an online prediction rule which a) adaptively achieves a MSE 
bound similar to the oracle prediction rule b) is computationally efficient. In the batch set up, this question is 
classical (especially in the univariate setting when $d = 1$) and has recently been studied thoroughly in 
general dimensions in~\cite{chatterjee2019adaptive}. It has been shown there that the Dyadic CART 
estimator achieves this near (up to log factors) oracle performance when $d \leq 2$ and a more 
computationally intensive version called the Optimal Regression Tree estimator (ORT) can achieve this near 
oracle performance in all dimensions under some assumptions on the true underlying partition. 
However, we are not aware of this question 
being explicitly answered in the online setting. We now state a theorem saying that the OMADRE 
essentially attains this objective. Below we denote the set of all partitions of $L_{d,n}$ into rectangles by 
$\mathcal{P}_{\all}.$ Note that the set $\mathcal{P}_{\dpt}$ is strictly contained in the set $\mathcal{P}_{\all}$.

\begin{theorem}[Oracle Inequality for Arbitrary Rectangular Partitions]\label{thm:pcconst}
	Let $\mathsf T$ be any subset of $K$ and $\hat{\theta}^{OM}$ denote the OMADRE predictor. There exists an 
	absolute constant $C$ such that for any $\lambda \ge C (\sigma \sqrt{\log N} 
	\vee\|\theta^*\|_{\infty})$,  one has for any non-anticipating ordering $\rho$ of $L_{d, n}$, 
	\begin{align}\label{eq:pcconst}
		\E\, \frac{1}{|\mathsf T|}\|\hat \theta^{OM}_{\mathsf T} - \theta_{\mathsf T}^* \|^2  \le 
		\inf_{\substack{P \in {\mathcal P}_{\all, \mathsf T}\\ \theta \in \Theta_P \subset \R^{\mathsf T}}} 
		&\big( \frac{1}{|\mathsf T|} \|\theta_{\mathsf T}^* - \theta \|^2 + C \lambda^2 
		\frac{|P|}{|\mathsf T|} (\log \e n)^d \log 2^d N \big) + \frac{\sigma^2 + \lambda^2}{|\mathsf T|^2}.
	\end{align}
	\end{theorem} 
	
We now discuss some noteworthy aspects of the above theorem.
\begin{enumerate}
	\item It is worth emphasizing that the above oracle inequality holds over all subsets $\mathsf T$ of $K$ simultaneously. Therefore, the OMADRE is a \textit{spatially adaptive}
	estimator in the sense of Section~\ref{sec:adap}. Such a guarantee is not available for any existing estimator, even in the batch learning setup. 
	For example, in the batch learning setup, all available oracle risk bounds for estimators such as Dyadic 
	CART and related variants are known only for the full sum of squared errors over the entire domain. 
	
	\medskip
	
	\item To the best of our knowledge, the above guarantee is the first of its kind explicitly stated in the 
	online learning setup. Therefore, the above theorem shows it is possible to attain a near (up to log factors) 
	oracle performance by a near linear time computable estimator in the online learning set up as well; thereby answering our first main question laid out in Section~\ref{sec:intro}.
	
	\medskip

	\item We also reiterate that the infimum in the R.H.S in~\eqref{eq:pcconst} is over the space of {\em all} 
	rectangular partitions $\mathcal{P}_{\all}.$ This means that if the true signal $\theta^*$ is piecewise 
	constant on an {\em arbitrary}  rectangular partition with $k$ rectangles, the OMADRE attains the 
	desired $\tilde{O}(k/N)$ rate. Even in the batch learning set up, it is not known how to attain this rate in 
	full generality. For example, it has been shown in~\cite{chatterjee2019adaptive} that the Dyadic 
	CART (or ORT) estimator enjoys a similar bound where the infimum is over the space of all recursive dyadic rectangular partitions (respectively decision trees) of $K$ which is a {\em stirct} subset of $\mathcal P_{{\rm all}}$. Thus, the bound presented here is stronger in this sensethan both these bounds known for Dyadic CART/ORT. More details about comparisons with Dyadic CART and ORT is given in Section~\ref{sec:discuss}.
	
	\medskip
	
\end{enumerate}
	
	
\subsection{\textbf{Result for Functions with Bounded Total Variation in General Dimensions}}
\label{sec:TV}
Consider the function class whose total variation (defined below) is bounded by some number. This is a 
classical function class of interest in offline nonparametric regression since it contains functions which 
demonstrate spatially heterogenous smoothness; see Section $6.2$ in~\cite{tibshiraninonparametric} and 
references therein. In the offline setting, the most natural estimator for this class of functions is what is 
called the Total Variation Denoising (TVD) estimator. The two dimensional version of this estimator is also 
very popularly used for image denoising; see~\cite{rudin1992nonlinear}. It is known that a well tuned TVD 
estimator is minimax rate optimal for this class in all dimensions; see~\cite{hutter2016optimal} 
and~\cite{sadhanala2016total}.

In the online setting, to the best of our knowledge, the paper~\cite{baby2019online} gave the first online algorithm attaining the minimax optimal rate. This algorithm is based on wavelet shrinkage. Recently, the paper~\cite{baby2021optimal} studied a version of the 
OMADRE in the context of online estimation of univariate bounded variation functions. In this  section we state a result showing that with our definition of the OMADRE, it is possible to predict/forecast bounded 
variation functions online in general dimensions at nearly the same rate as is known for the batch set up.

We can think of $K = L_{d,n}$ as the $d$ dimensional regular lattice graph. Then, thinking of $\theta \in \R^{L_{d,n}}$ as a function on $L_{d,n}$ we define
\begin{equation}\label{eq:TVdef}
\TV(\theta) =  \sum_{(u,v) \in E_{d,n}} |\theta_{u} - \theta_{v}| 
\end{equation}
where $E_{d,n}$ is the edge set of the graph $L_{d,n}$. The above definition can be motivated via the analogy with the continuum case. If we think  
$\theta[i_1,\dots,i_n] = f(\frac{i_1}{n},\dots,\frac{i_d}{n})$ for a differentiable function 
$f: [0,1]^{d} \rightarrow \R$, then the above definition divided by $n^{d - 1}$ is precisely 
the Reimann approximation for $\int_{[0,1]^d} \|\nabla f\|_1$. In the sequel we denote,
\begin{equation*}
	\mathcal{BV}_{d, n}(V^*) \coloneqq 	\{\theta \in \R^n : \TV(\theta) \le V^*\}.
\end{equation*}

We are now ready to state:
\begin{theorem}[Prediction error for $\mathcal{BV}_{d, n}(V^*)$ with online averages]
	\label{thm:TV_slow_rate}	
	Fix any $\mathsf T \subset K$ that is a dyadic square and denote $V^*_{\mathsf T} = TV(\theta^*_{\mathsf T})$. If $\lambda \ge C (\sigma \sqrt{\log N} 
	\vee\|\theta^*\|_{\infty})$ as in Theorem~\ref{thm:generic_online}, we have for some 
	absolute constant $C > 1$ and any non-anticipating ordering $\rho$ of $L_{d, n}$,
	\begin{equation}\label{eq:TV_slow_rate_d>1}
	\E \, \frac{1}{|\mathsf T|} \|\hat \theta^{OM}_{\mathsf T} - \theta_{\mathsf T}^* \|^2 \le 
	\frac{C}{|\mathsf T|} \big(\lambda^2 (\log 2^d N)^2 + \lambda V_{\mathsf T}^* (\log 2^d N)^{3/2} \big) + 
	\frac{\sigma^2  + \lambda^2}{|\mathsf T|^2}.
	\end{equation}
	when $d > 1$. On the other hand, for $d = 1$ we have
	\begin{equation}\label{eq:TV_slow_rate_d=1}
	\E \, \frac{1}{|\mathsf T|} \|\hat \theta^{OM}_{\mathsf T} - \theta_{\mathsf T}^* \|^2 \le 
	C\, \lambda^{4/3} (\log 2^d N)^{4/3} \left(\frac{V_{\mathsf T}^*}{|\mathsf T|}\right)^{2/3} + \, 
	\frac{\sigma^2 + \lambda^2}{|\mathsf T|^2}.
	\end{equation} 
\end{theorem}



Here are some noteworthy aspects of the above theorem.

\begin{enumerate}
	\item The above theorem ensure that the OMADRE matches the known minimax rate of 
	estimating bounded variation functions in any dimension. To the best of our knowledge, this result is new 
	in the in the online setting for the multivariate (i.e., $d \geq 2$) case.
	
	\item Note that our MSE bounds hold simultaneously for all dyadic square regions. Thus, the OMADRE adapts to the unknown variation of the signal $V^*_{\mathsf T}$, for any local dyadic square 
	region $\mathsf T$.  In this sense, the OMADRE is spatially adaptive. Even in the batch setting, 
	this type of simultaneous guarantee over a class of subsets of $L_{d,n}$ is not available for the canonical 
	batch TVD estimator.

	\item We require $\mathsf T$ to be a dyadic square because of a particular step in our proof where we 
	approximate a bounded variation array with an array that is piecewise constant over a recursive 
	dyadic partition of $L_{d,n}$ with pieces that have bounded aspect ratio. See 
	Proposition~\ref{prop:division} in the appendix.
	\end{enumerate}

\section{Online Linear Regression Aggregation over Dyadic Rectangles (OLRADRE)}\label{sec:olradre}
In this section, we consider another instantiation of our general prediction algorithm which 
is based on the Vovk, Azoury and 
Warmuth online linear regression forecaster, e.g see~\cite{vovk1998competitive}. Similar to Section~\ref{sec:omadre}, we take our set of experts $\mathcal{S}$ to 
be the set of all dyadic sub-rectangles of $L_{d, n}$. However, the main difference is that 
we now take $F_S$ to be the subspace spanned by a finite set $\mathcal F$ of basis functions on $\R^d$ 
restricted to $S$ .  In the next two subsections, we will focus specifically on the case when $\mathcal F$ is the set of all monomials in $d$ 
variables with a maximum degree (see \eqref{def:F_piecepoly} below).

Next we need to choose an online rule $\mb r$ to which end the VAW linear regression forecaster leads to an online rule defined as:
\begin{equation}\label{def:online_regression}
	r^{(S)}_{U, s}((y_u : u \in U)) = \hat \beta_s \cdot x_s \mbox{ with } \hat \beta_s \coloneqq 
	\Big(I + \sum_{u \in U \cup \{s\}} x_u x_u^T \Big)^{-1} \big(\sum_{u \in U} y_u x_u \big)
\end{equation} for all $U \subset S \in \mathcal S$ and $s \in S \setminus U$ where $x_u$ is the vector 
$(f(u):  f \in \mathcal F) \in \R^{\mathcal F}$ and $\cdot$ denotes the canonical inner product in 
$\R^{\mathcal F}$. By convention, we interpret an empty summation as 0.

The next lemma gives the computational complexity of the OLRADRE  which is the same as that of the OMADRE except that it scales cubically with the cardinality of the basis function class $\mathcal{F}.$
\begin{lemma}[Computational complexity of {\rm OLRADRE}]\label{lem:complexity_online_regress}
There exists an  absolute constant $C > 0$ such that the computational complexity of OLRADRE is bounded by $C |\mathcal F|^3 N (\log_2 2n)^d$.
\end{lemma}



We reiterate here that the set of basis functions can be taken to be anything (e.g relu functions, wavelet basis etc.) and yet a simultaneous oracle risk bound such as Theorem~\ref{thm:generic_online} will hold for 
the OLRADRE. We now move on to focus specifically on piecewise polynomial and 
univariate higher order bounded variation functions where OLRADRE performs near optimally.


\subsection{\textbf{Result for Rectangular Piecewise Polynomial Functions in General Dimensions}}\label{sec:piece_poly}
The setup for this subsection is essentially similar to that in Section~\ref{sec:piecewise_constant} except 
that $\theta^*$ can now be piecewise {\em polynomial} of degree at most $m$ on the (unknown) partition 
$P^*$.  More precisely, we let
\begin{equation}\label{def:F_piecepoly}
	\mathcal F \coloneqq  \Big\{ \Big(\frac{u}{n}\Big)^{\mb m} \coloneqq \prod_{i \in [d]} \Big(\frac{u_i}{n}\Big)^{m_i}: |\mb m| 
	\le m \Big\} 
\end{equation}
where $\mb m = (m_1, \ldots, m_d)$ is the {\em multidegree} of the monomial $\Big(\frac{u}{n}\Big)^{\mb 
	m}$ and $|\mb m| = \sum_{i \in [d]} m_i$ is the corresponding degree.

As before, an \textit{oracle predictor} $\hat{\theta}_{({\rm oracle})}$ --- which {\em knows} the minimal rectangular partition $(R_1,\dots,R_k)$ of $\theta^*$ {\em exactly} ---can simply use the online VAW online linear regression
rule given in~\eqref{def:online_regression} separately within each of the rectangles$(R_1,\dots,R_k).$ By a 
basic result about VAW online linear regression, (see Proposition~\ref{prop:online_regress_regret}), it can be shown 
that the MSE of this oracle predictor is bounded by $O_{d}(\frac{k \|\theta^*\|_{\infty} \log n}{N}).$ In words,  
the MSE of this oracle predictor again scales (up to a log factor which is necessary) like the \textit{number of 
	constant pieces of $\theta^*$ divided by the sample size $N$ which is precisely the parametric rate of convergence.} We will now state a result saying that OLRADRE, which is computationally efficient, can attain this oracle rate of convergence, up to certain additional multiplicative log factors.

Since any $\theta \in  \Theta_P$, where $P \in \mathcal P_{\all, \mathsf T}$ for some $\mathsf T \subset K$ 
(cf.~ the statement of Theorem~\ref{thm:pcconst}), is piecewise polynomial on $P$, we can associate to any 
such $\theta$ the number
\begin{equation}\label{def:sobolev_norm}
s_{m, \infty}(\theta) = \max_{S \in P,\, |\mb m| \le m} n^{|\mb m|}\beta_{\mb m, S} \mbox{ where } 
\theta_S \equiv \sum_{|\mb m| \le m} \beta_{\mb m, S} \, u^{\mb m}  =  \sum_{|\mb m| \le m} n^{|\mb m|} 
\beta_{\mb m, S} \, \Big(\frac{u}{n}\Big)^{\mb m}.\end{equation}
The reader should think of $\theta_s$ as $g( \frac sn)$ where $g$ is some piecewise polynomial function 
defined on the unit cube $[0, 1]^d$  and hence of $s_{m, \infty}(\theta)$ as its maximum coefficient which is a bounded number, i.e., it does not grow with $n$. Let us keep in mind that the OLRADRE depends on the underlying degree $m$ which we keep implicit in our discussions  below. We can now state the 
analogue of Theorem~\ref{thm:pcconst} in this case.


\begin{theorem}[Oracle Inequality for Arbitrary Rectangular Partitions]\label{thm:pcpoly}
Let $\mathsf T$ be any subset of $K$ and $\hat{\theta}^{OL}$ denote the OLRADRE predictor. Then there 
exist an absolute constant $C$ and a number $C_{m, d} > 1$ depending only on $m$ and $d$ such that for 
$\lambda \ge C (\sigma \sqrt{\log N} \vee \|\theta\|_{\infty})$, one has for any non-anticipating ordering $\rho$ of $L_{d, n}$,
\begin{align}\label{eq:generic_online_regr2}
\E \,\frac{1}{|\mathsf T|}\|\hat \theta^{OL}_{\mathsf T} - \theta_{\mathsf T}^* \|^2  \le \inf_{\substack{P \in {\mathcal 
P}_{\all,\mathsf T}\\ \theta \in \Theta_P \subset \R^{\mathsf T}}} &\big( \frac{1}{|\mathsf T|} \|\theta_{\mathsf 
T}^* - \theta \|^2 + C_{m,d} \lambda_{m, *}^2 \frac{|P|}{|\mathsf T|} (\log \e n)^d \log 2^d N \big) + 
\frac{\sigma^2 + \lambda^2}{|\mathsf T|^2}
	\end{align}
where $\lambda_{m, *} = \lambda + s_{m, \infty}(\theta)$.
\end{theorem}

We now make some remarks about this theorem.

\begin{remark}
	We are not aware of such a simultaneous oracle risk bound explicitly stated before in the literature for piecewise polynomial signals in general dimensions in the online learning setting. 
	\end{remark}

\begin{remark}
	Even in the batch learning setting, the above oracle inequality is a stronger result than available results for higher order Dyadic CART or ORT~(\cite{chatterjee2019adaptive}) in the sense that the infimum is taken over the space of all rectangular partitions $\mathcal{P}_{\all}$ instead of a more restricted class of partitions. 
	\end{remark}

\subsection{\textbf{Result for Univariate Functions of Bounded Variation of Higher Orders}}\label{sec:univar}
One can consider the univariate function 
class of all $m$ times (weakly) differentiable functions, whose $m$ th derivative is of bounded variation. This is also a canonical function class in offline nonparametric regression. A seminal result of~\cite{donoho1998minimax} shows that a wavelet threshholding estimator attains the minimax rate in this problem. Locally adaptive regression splines, proposed by~\cite{mammen1997locally}, is also known to achieve the minimax rate in this problem. 
Recently, Trend Filtering, proposed by~\cite{kim2009ell_1}, has proved to be a popular nonparametric regression method. Trend Filtering is very closely related to locally adaptive regression splines and is also minimax rate optimal over the space of higher order bounded variation functions; see~\cite{tibshirani2014adaptive} and references therein. Moreover, it is known that Trend Filtering adapts to functions which are piecewise polynomials with regularity at the knots. If the number of pieces is not too large and the length of the pieces is not too small, a well tuned Trend Filtering estimator can attain near 
parametric risk as shown in~\cite{guntuboyina2020adaptive}. In the online learning setting, this function 
class has been studied recently by \cite{baby2020adaptive} using online wavelet shrinkage methods. We 
now state a spatially adaptive oracle risk bound attained by the OLRADRE  for this function class.

Let $K = L_{1, n} = [[1, n]]$ and for any vector $\theta \in \R^n$, let us define its $m$-th 
order (discrete) derivative for any integer $r \ge 0$ in a recursive manner as follows. We start with $D^{(0)}(\theta) = \theta$ and $D^{(1)}(\theta) = (\theta_2 - 
\theta_1,\dots,\theta_n - \theta_{n - 1})$. Having defined $D^{(m-1)}(\theta)$ for some $m \ge 
2$, we set $D^{(m)}(\theta) = D^{(1)}(D^{(m - 1)}(\theta))$. Note that $D^{(m)}(\theta) \in 
\R^{n - m}$. For sake of convenience, we denote the operator $D^{(1)}$ by $D$. For any positive integer 
$m \geq 1$, let us also define the $m$-th order variation of a vector $\theta$ as follows:
\begin{equation}\label{eq:rth_orderTV}
V^{(m)}(\theta) = n^{m - 1} |D^{(m)}(\theta)|_{1}
\end{equation}
where $|.|_1$ denotes the usual $\ell^1$-norm of a vector. Notice that $V^{1}(\theta)$ is the total variation of a vector defined in~\eqref{eq:TVdef}. Like our definition of total 
variation, our definition in \eqref{eq:rth_orderTV} is also motivated by the analogy with the 
continuum. If we think of $\theta$ as an evaluation of an $m$ times differentiable function 
$f:[0,1] \rightarrow \R$ on the grid $(1/n,2/n\dots,n/n)$, then the Reimann approximation to 
the integral $\int_{[0,1]} f^{(m)}(t) dt$ is precisely equal to $V^{(m)}(\theta)$. Here 
$f^{(m)}$ denotes the $m$-th order derivative of $f$. Thus, the reader should assume that 
$V^{(m)}(\theta)$ is of constant order for a generic $\theta$. 
Analogous to the class $\mathcal{BV}_{d, n}(V^*)$, 
let us define for any integer $m \ge 1$,
\begin{equation*}
	\mathcal{BV}^{(m)}_{n}(V^*) = \{\theta \in \R^n:  V^{(m)}(\theta)\leq V^*\}.
\end{equation*}

In the spirit of our treatment of the class $\mathcal{BV}_{d, n}(V^*)$ in Section~\ref{sec:TV}, we take 
$$\mathcal F = \{1, x, \ldots, x^{m-1}\}.$$  We now state the main result of this subsection.
\begin{theorem}[Prediction error for $\mathcal{BV}_{N}^{(m)}(V^*)$, $m > 1$]
	\label{thm:trendfilter_slow_rate}	
	Fix any interval $\mathsf T \subset K$ and denote $V_{\mathsf T}^* = V^{(m)}(\theta^*_{\mathsf T})$. Also let $\|\theta^*\|_{m-1, \infty} \coloneqq \max_{0 \le j < m}N^{j}\|D^{j}(\theta^*)\|_{\infty}$. 
	Then there exist an absolute constant $C$ and a number $C_m > 1$ depending only on $m$ such that for $\lambda \ge C (\sigma \sqrt{\log N} \vee \|\theta^*\|_{\infty})$, we have 
	for any non-anticipating ordering $\rho$ of $L_{1, n}$,
	\begin{equation}\label{eq:trendfilter_slow_rate_d=1}
		\E \, \frac{1}{|\mathsf T|}\| \hat \theta_{\mathsf T }^{OL}  - \theta_{\mathsf T}\|^2 \le C_m\, \lambda_{m, *} ^{\frac{4m}{2m + 1}} \log \e N 
		\left(\frac{(V_{\mathsf T}^*)^{1/m}}{|\mathsf T|}\right)^{\frac{2m}{2m + 1}} + \, \frac{\sigma^2 + \lambda^2}{|\mathsf T|^2}
	\end{equation} 
where $\lambda_{m, *} \coloneqq \lambda + \|\theta^*\|_{m-1, \infty}$ (cf.~the statement of 
Theorem~\ref{thm:pcpoly}).
\end{theorem}

We now make some remarks about the above theorem.

\begin{remark}
	The above spatially adaptive risk bound for bounded variation functions of a general order is new even in the easier batch learning setting. State of the art batch learning estimators like Trend Filtering or Dyadic CART are not known to attain such a spatially adaptive risk bound. 
\end{remark}

%

\section{Discussion}
In this section we discuss some natural related matters.

\subsection{Detailed Comparison with Dyadic CART} \label{sec:discuss}

The Dyadic CART is a natural offline analogue of the OMADRE described in 
Section~\ref{sec:omadre}. Similarly, higher order versions of Dyadic CART and Trend Filtering are natural 
offline analogues of the univariate piecewise polynomial OLRADRE  described in 
Section~\ref{sec:olradre}. Therefore, it makes sense to compare our oracle risk bound (notwithstanding 
simultaneity and the fact that OMADRE/OLRADRE are online algorithms) in 
Theorems~\ref{thm:pcconst},~\ref{thm:pcpoly},~\ref{thm:trendfilter_slow_rate} with the available offline 
oracle risk bound for Dyadic CART, see Theorem $2.1$ in~\cite{chatterjee2019adaptive}.
This result is an oracle risk bound  where the infimum is over all recursive dyadic partitions (see a precise definition in Section of~\cite{chatterjee2019adaptive}) of $L_{d,n}.$ On the other hand,  our oracle risk bounds are essentially an infimum over all dyadic partitions $\mathcal{P}_{dp}.$ In dimensions $d = 1,2$ these two classes of partitions coincide (see Lemma $8.2$ in~\cite{chatterjee2019adaptive}) but for $d > 2$, the class of partitions $\mathcal{P}_{dp}$ strictly contain the class of recursive dyadic partitions (see Remark $8.3$ in~\cite{chatterjee2019adaptive}). Therefore, the oracle risk bounds in Theorems~\ref{thm:pcconst},~\ref{thm:pcpoly} are stronger in this sense.

The above fact also allows us to convert the infimum over all dyadic partitions $\mathcal{P}_{dp}$ to the space of all rectangular partitions $\mathcal{P}_{\all}$ since any partition in $\mathcal{P}_{\all}$ can be refined into a partition in $\mathcal{P}_{dp}$ with the number of rectangles inflated by a $(\log n)^d$ factor. 
In dimensions $d \geq 3$, such an offline oracle risk bound (where the infimum is over $\mathcal{P}_{\all}$) is not known for Dyadic CART. As far as we are aware, the state of the art result here is shown in~\cite{chatterjee2019adaptive} where the authors show that a significantly more computationally intensive version of Dyadic CART, called the ORT estimator is able to adaptively estimate signals which are piecewise constant on fat partitions. In contrast, Theorems~\ref{thm:pcconst},~\ref{thm:pcpoly} hold for all dimensions $d$, the infimum in the oracle risk bound is over the set of all rectangular partitions $\mathcal{P}_{\all}$ and no fatness is needed.

It should also be mentioned here  that compared to batch learning bounds for Dyadic CART, our bounds have an extra log factor and some signal dependent factors which typically scale like $O(1).$ Note that the computational complexity of our algorithm is also worse by a factor $(\log n)^d$, compare Lemma~\ref{lem:complexity_online_mean} to Lemma $1.1$ in~\cite{chatterjee2019adaptive}. However, it should be kept in mind that we are in the online setup which is a more difficult problem setting than the batch learning setting.

\subsection{Some Other Function Classes} \label{sec:otherfunctons}
Our simultaneous oracle risk bounds are potentially applicable to other function classes as well not considered in this paper. We now mention some of these function classes.

A similar batch learning oracle risk bound with an infimum over the set of all recursive dyadic partitions was used by Donoho (1997) to demonstrate minimax
rate optimality of Dyadic CART for some anisotropically smooth bivariate function classes. Using our result, it should be possible to attain a simultaneous version of minimax rate optimal bounds for these types of function classes.

Consider the class of bounded monotone signals on $L_{d,n}$ defined as
\begin{equation*}
	\mathcal{M}_{d,n} = \{\theta \in [0,1]^{L_{n,d}}: \theta[i_1,\dots,i_d] \leq \theta[j_1,\dots,,j_d] \:\:\text{whenever}\:\: i_1 \leq j_1,\dots,i_d \leq j_d\}.
\end{equation*}
Estimating signals within this class falls under the purview of Isotonic Regression. Isotonic Regression has been a topic of recent interest in the online learning community; see~\cite{kotlowski2016online},~\cite{kotlowski2017random}. It can be checked that the total variation for any $d$ dimensional isotonic signal with range $O(1)$ grows like $O(n^{d - 1})$ which is of the same order as a canonical bounded variation function. Therefore, the bound in Theorem~\ref{thm:pcconst} would give spatially adaptive minimax rate optimal bounds for Isotonic Regression as well. In the offline setup, a lot of recent papers have investigated Isotonic regression with the aim of establishing minimax rate optimal rates as well as near optimal adaptivity to rectangular piecewise constant signals; see~\cite{deng2020isotonic},~\cite{han2019isotonic}. Theorem~\ref{thm:pcconst} establishes that such adaptivity to rectangular piecewise constant signals as well as maintaining rate optimality over isotonic functions is also possible in the online setting by using the OMADRE proposed here.

Let us now consider univariate convex regression. In the offline setting, it is known that the least squares estimator LSE is minimax rate optimal, attaining the $\tilde{O}(n^{-4/5})$ rate, over convex functions with bounded entries, see e.g.~\cite{GSvex},~\cite{chatterjee2016improved}. It is also known that the LSE attains the $\tilde{O}(k/n)$ rate if the true signal is piecewise linear in addition to being convex. Theorem~\ref{thm:pcpoly} and Theorem~\ref{thm:trendfilter_slow_rate} imply both these facts also hold for the OLRADRE (since a convex function automatically has finite second order bounded variation) where we fit linear functions (polynomial of degree $1$) on intervals. To the best of our knowledge, such explicit guarantees for online univariate convex regression were not available in the literature before this work.

\subsection{Computation Risk Tradeoff}
The main reason for us considering dyadic rectangles (instead of all rectangles) as experts is to save computation. In particular, if one uses the set of all rectangles as experts, the computational complexity of the resulting algorithm would be $O_{d}(N^3)$. One can think of this estimator as the online analogue of the ORT estimator defined in~\cite{chatterjee2019adaptive}. For this estimator, the risk bounds would be better. For example, the $(\log n)^d$ term multiplying $|P|$ in the bound in Theorems~\ref{thm:pcconst},~\ref{thm:pcpoly} would now no longer be present. In particular, the exponent of $\log n$ would be $2$ for all dimensions $d$ which is only one log factor more than a known minimax lower bound for the space of all rectangular piecewise constant functions; see Lemma $3.1$ in~\cite{chatterjee2019adaptive}.

One can also easily interpolate and take the set of experts somewhere between the set of dyadic rectangles and the set of all rectangles, say by considering all rectangles with side lengths a multiple of some chosen integer $l$. Thus one can choose the set of experts by trading off computational time and the desired statistical prediction performance.

\subsection{Open Problems}
In our opinion, our work here raises some interesting open questions which we leave for future research. 
\begin{enumerate}
	\item It appears that if a function class is well approximable by rectangular piecewise constant/polynomial functions then the type of oracle risk bounds proved here may be used to derive some nontrivial prediction bounds. However, for many function classes, this kind of approximability may not hold. For example, we 
	can consider the class of Hardy Krause Bounded Variation Functions~(see~\cite{fang2021multivariate}) or its higher order versions~(see~\cite{ki2021mars}) where the existing covering argument produces nets (to estimate metric entropy) which are not necessarily rectangular piecewise constant/linear respectively. 
	These function classes are also known not to suffer from the curse of dimensionality in the sense that the metric entropy does not grow exponentially in $\frac{1}{\epsilon}$ with the dimension $d$. More generally, it would be very interesting to come up with computationally efficient and statistically rate optimal online 
	prediction algorithms for such function classes. 
	
	\medskip

	\item The analysis presented here relies a lot on the light tailed nature of the noise. It can be checked that Theorem~\ref{thm:generic_online} can also be proved when the noise is mean $0$ sub exponential, we would only get an appropriate extra log factor. However, the proof would break down for heavy tailed noise. This seems to be an open area and not much attention has been given to the noisy online prediction problem with heavy tailed noise. Most of the existing results in the online learning community assume bounded but arbitrary data. The heavy tailed setting we have in mind is that the data $y$ is not arbitrary but of the form signal plus noise, except that the noise can be heavy tailed. It would be very interesting to obtain an analogue of Theorem~\ref{thm:generic_online} in this setting. Clearly, the algorithm has to change as well in the sense that instead of aggregating means one should aggregate medians of various rectangles in some appropriate way. 
	
		\medskip

\item Another important aspect that we have not discussed here is the issue of choosing the 
tuning/truncation parameter $\lambda$ in a {\em data driven} manner. It is possibly natural to choose a grid 
of candidate truncation values and run an exponentially weighted aggregation algorithm aggregating the 
predictions corresponding to each truncation value. This approach was already considered in 
\cite{baby2021optimal} (see Section~4). However, since our data is unbounded, we run into the same issue 
of choosing an appropriate tuning parameter. It is an important research direction to investigate whether the 
recent developments in the cross validation there are any other natural ways to address this problem.

\end{enumerate}

\section{Simulations}

\subsection{1D Plots}

We provide plots of the OMADRE for a visual inspection of its performance. There are three plots for scenarios corresponding to different true signals $\theta^*$, where for any $i \in [n]$, we have $\theta^*_i = f(i/n)$ for some function $f : [0,1] \to \R$, specified below and the errors are generated from $N(0, 1)$. The sample size is taken to be $n = 2^{16}$ for these plots, given in Figure~\ref{fig:tv1}. The truncation parameter $\lambda$ has been taken to be $2 \max\{\|\theta^*\|_{\infty}, \sigma \left(2 \log n\right)^{1/2}\}$ for all our $1D$ simulations. It may be possible to get better predictions by choosing a smaller value of $\lambda$ but we have not done any systematic search for these simulations as this particular choice seemed to work well.

The ordering of the revealed indices is taken to be the forward ordering $1,2,3,\dots$ and the backward ordering $n,n - 1,n - 2,\dots$. The predictions corresponding to the two orderings are then averaged in the plots. 



\begin{enumerate}
	\item Scenario 1 [Piecewise Constant Signal]: We consider the piecewise constant function 
	$$f(x) = 2(\mathrm{1}(x \in [1/5, 2/5])) + \mathrm{1}(x \in [2/5, 3/5]) + 2\mathrm{1}(x \in [3/5, 4/5]),$$
	and consider the the 1D OMADRE. The corresponding plot is shown in the second diagram of Figure~\ref{fig:tv1}.
	\item Scenario 2 [Piecewise Linear Signal]: We consider the piecewise linear function 
	$$f(x) = 6x(\mathrm{1}(x \in [0, 1/3])) + (-12x + 6)\mathrm{1}(x \in [1/3, 2/3]) + (x - 8/3)(\mathrm{1}(x \in [2/3, 1])),$$
	and consider the 1D OMADRE. The corresponding plot is shown in the second diagram of Figure~\ref{fig:tv1}.
	\item Scenario 3 [Piecewise Quadratic Signal]: We consider the piecewise quadratic function 
	$$f(x) = 
	\begin{cases}
	18x^2 \quad &\text{if}\;\; x \in [0, 1/3]\\
	-36(x-1/2-1/\sqrt{12})(x-1/2+\sqrt{12}) \quad &\text{if}\;\; x \in [1/3, 2/3]\\
	18(x-1)^2 \quad &\text{if}\;\; x \in [2/3, 1]
	\end{cases}.
	$$
	and consider the 1D OMADRE estimator. The corresponding plot is shown in the third diagram of Figure~\ref{fig:tv1}.
	
\end{enumerate}

\begin{figure}[H]
	\centering
	\includegraphics[scale = 0.25]{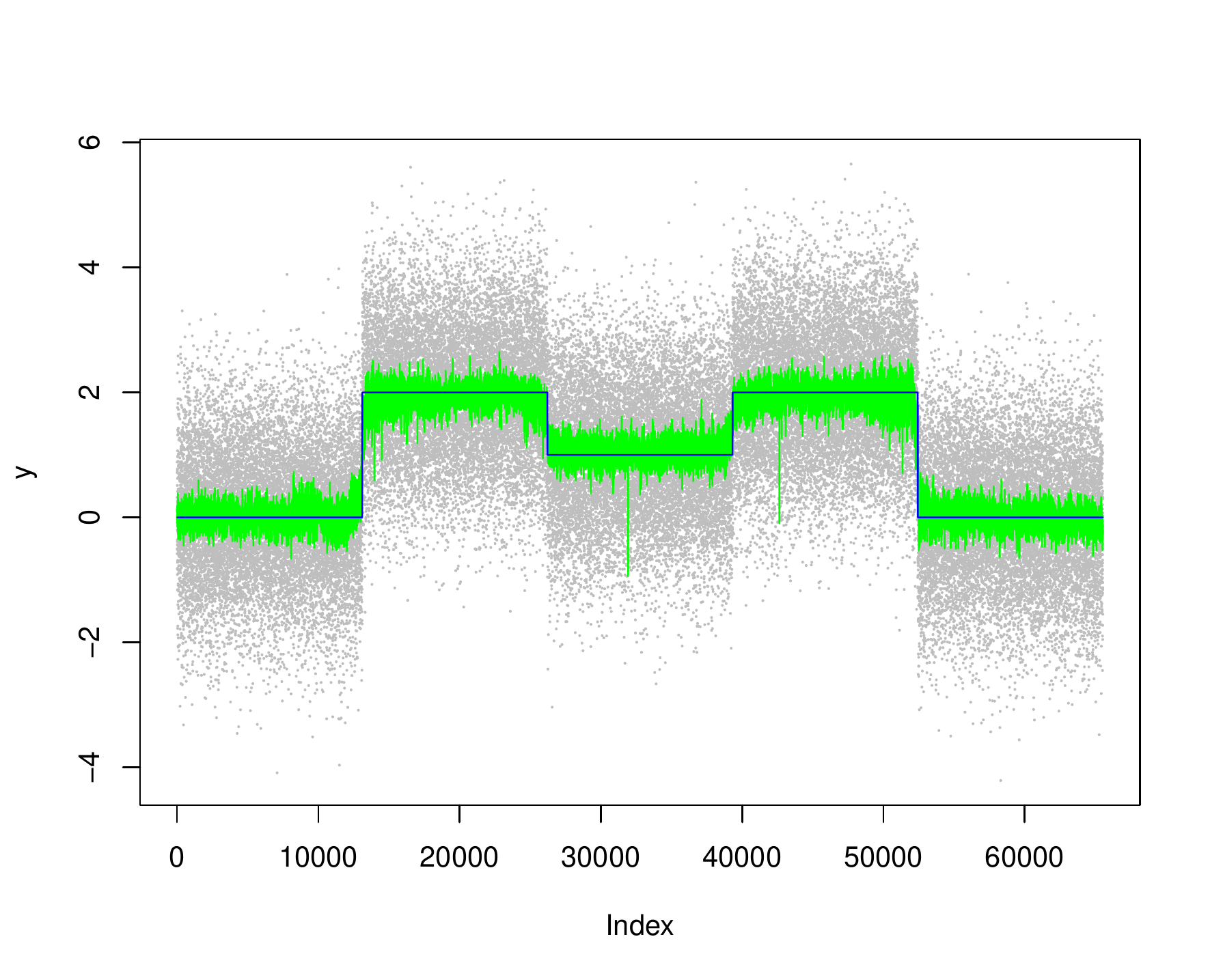} \includegraphics[scale = 0.25]{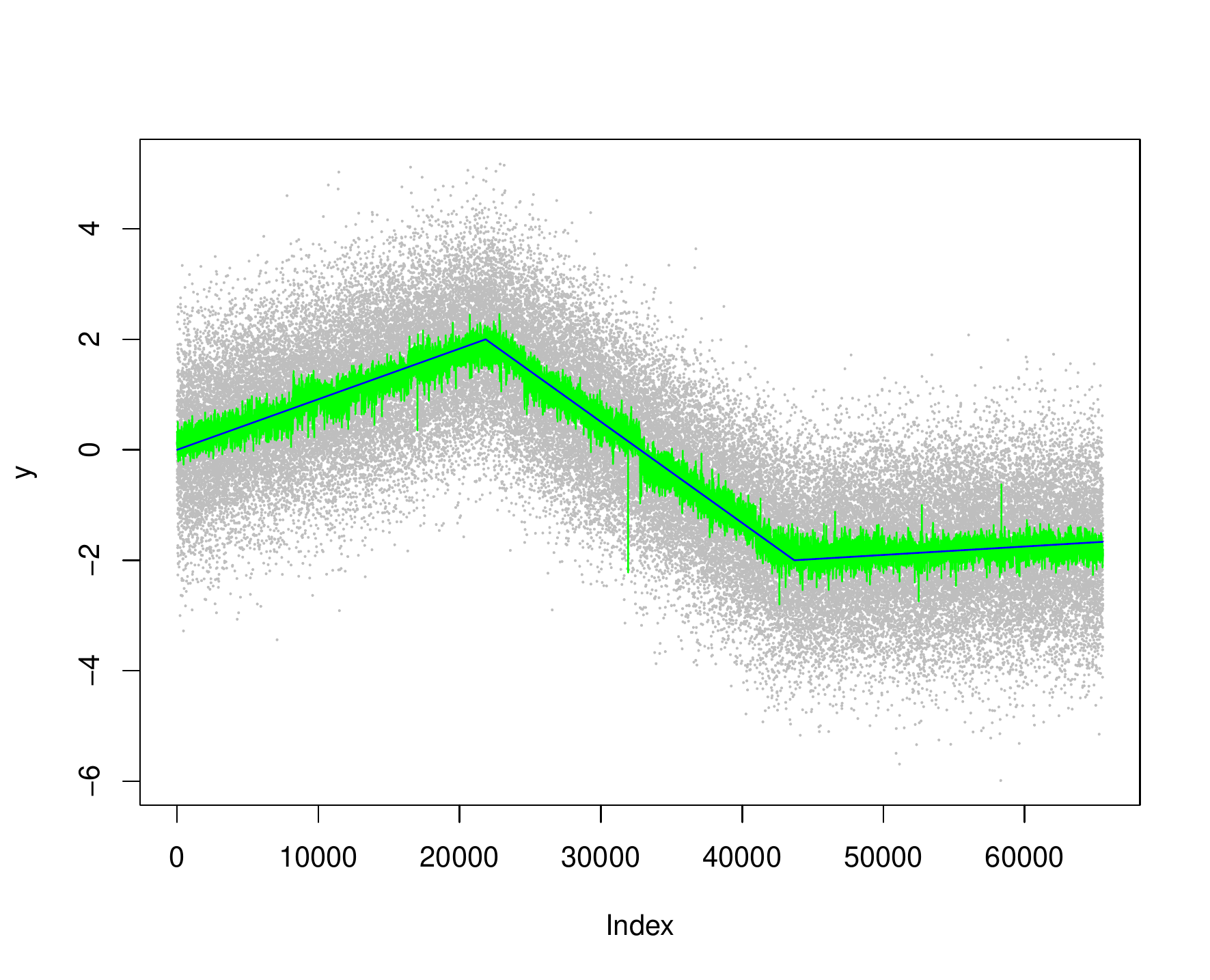} \includegraphics[scale = 0.25]{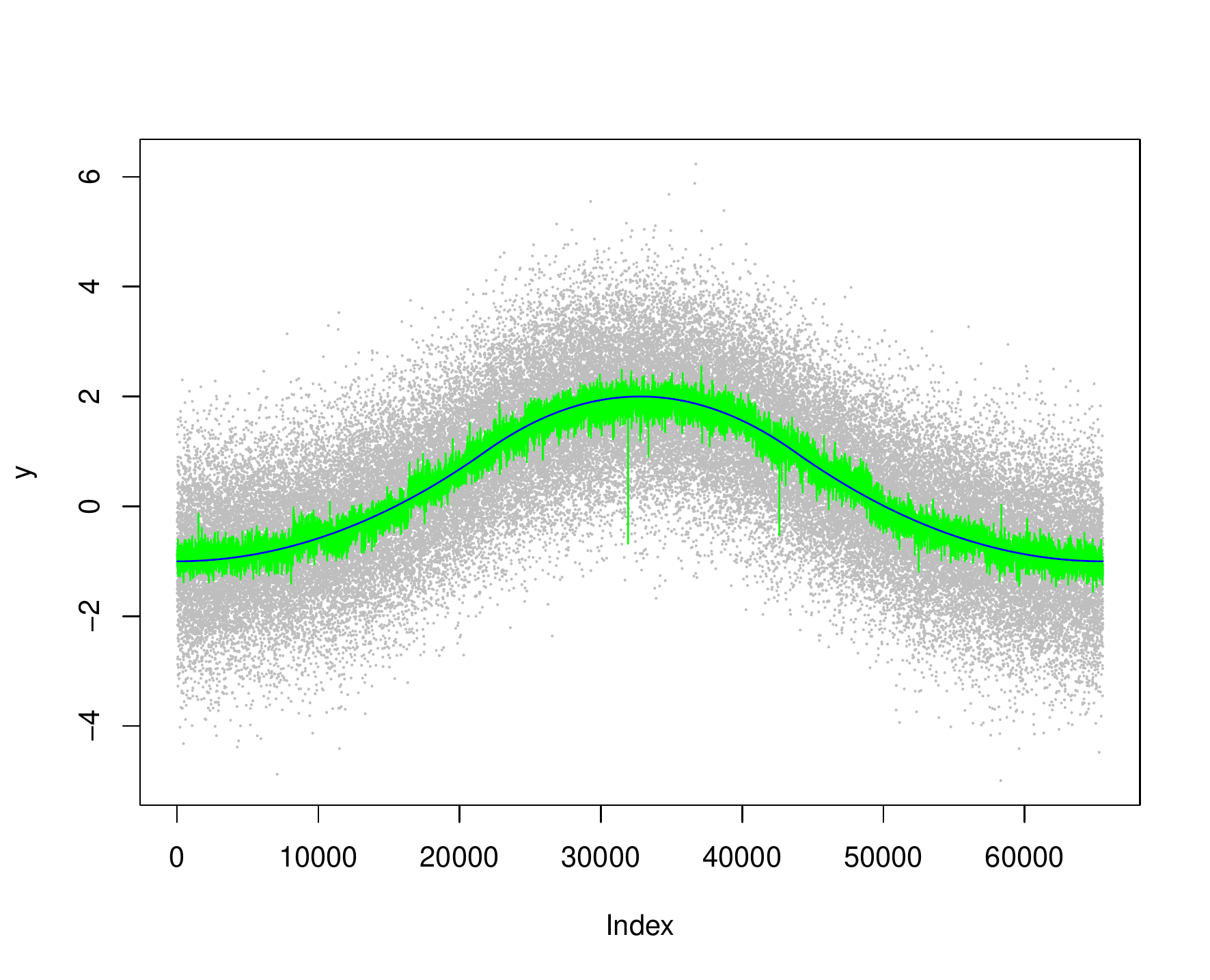}
	\caption{The blue curve is the true signal, the grey points are data points and the green curve constitutes the OMADRE predictions. The plotted predictions are averaged over two predictions when the data are revealed in the forward and backward order.}
	\label{fig:tv1}
\end{figure}

\subsection{1D Comparisons}
We conduct a simulation study to compare the performance of the OMADRE  and the OLRADRE of order $1,2$ which aggregates linear function predictions and quadratic function predictions. We consider the ground truth signal as the smooth sinusoidal function $$f(x) = \sin 2 \pi x + \cos 5 \pi x.$$ We considered various signal to noise ratios by setting the noise standard deviation $\sigma$ to be $0.5,1$ or $2$. We also considered sample sizes $n = 2^{10},2^{12},2^{14}.$ In each case, we estimated the MSE by 50 Monte Carlo replications. Here also, the predictions corresponding to the forward and backward orderings are averaged. We report the MSE's in Tables~\ref{tab:dc1},~\ref{tab:dc2} and~\ref{tab:dc3} respectively.

\begin{table}[H]
	\caption{MSEs of OMADRE estimator in different scenarios}
	\label{tab:dc1}
	\begin{tabular}{c|c|c|c}
		\hline
		$n$ &$\sigma = 0.5$  &$\sigma = 1$&$\sigma = 2$\\
		\hline
		$2^{10}$ & 0.045 & 0.076 & 0.191 \\
		$2^{12}$ & 0.027 & 0.048 & 0.127 \\
		$2^{14}$ & 0.014 & 0.031 & 0.087 \\
		\hline
	\end{tabular}
\end{table}

\begin{table}[H]
	\caption{MSEs of OLRADRE (linear) in different scenarios}
	\label{tab:dc2}
	\begin{tabular}{c|c|c|c}
		\hline
		$n$ &$\sigma = 0.5$  &$\sigma = 1$&$\sigma = 2$\\
		\hline
		$2^{10}$ & 0.088 & 0.099 & 0.143 \\
		$2^{12}$ & 0.049 & 0.057 & 0.087 \\
		$2^{14}$ & 0.025 & 0.030 & 0.050 \\
		\hline
	\end{tabular}
\end{table}

\begin{table}[H]
	\caption{MSEs of OLRADRE (quadratic) in different scenarios}
	\label{tab:dc3}
	\begin{tabular}{c|c|c|c}
		\hline
		$n$ &$\sigma = 0.5$  &$\sigma = 1$&$\sigma = 2$\\
		\hline
		$2^{10}$ & 0.079 & 0.091 & 0.136 \\
		$2^{12}$ & 0.040 & 0.048 & 0.079 \\
		$2^{14}$ & 0.020 & 0.025 & 0.044 \\
		\hline
	\end{tabular}
\end{table}

It is reasonable to expect that the OLRADRE  aggregating quadratic function predictions would perform no worse than the OLRADRE  aggregating linear function predictions which in turn would perform no worse than the OMADRE estimator. 
From the tables~\ref{tab:dc1},~\ref{tab:dc2} and~\ref{tab:dc3} we see that when the noise variance is low, the opposite happens and the OMADRE gives a better performance. It is only when the noise variance becomes high, the OLRADRE aggregating quadratic functions starts to perform the best. We see a similar phenomenon for other ground truth functions as well. We are not sure what causes this but we believe that in the low noise regime, the weights of the local experts are high (for the OMADRE estimator) and for smooth functions these predictions would be very accurate. Since the OLRADRE  has a shrinkage effect (note the presence of $I$ in the gram matrix), there is bias for the predictions of the local experts which is why the local experts in this case predict slightly worse than for the OMADRE estimator. In the case when the signal to noise ratio is low, the algorithms are forced to use experts corresponding to wider intervals for which case the bias of the OLRADRE predictions become negligible.

\subsection{2D Plots}
We conduct a simulation study to observe the performance of the proposed OMADRE estimator in three different scenarios each corresponding to a different true signal $\theta^*$. In every case, the errors are generated from a centered normal distribution with standard deviation $0.25$, the dimension $d = 2$ and we take the number of pixels in each dimension to be $n = 64, 128, 256$. We estimate the MSE by $50$ Monte Carlo replications and they are reported in Table~\ref{tab:omadre2d}. The truncation parameter $\lambda$ has been taken to be $2 \max\{\|\theta^*\|_{\infty}, \sigma \left(2 \log(n^2)\right)^{1/2}\}.$
In each of the cases, a uniformly random ordering of the vertices of $L_{2,n}$ has been taken to construct the OMADRE estimator. Overall, we see that our OMADRE estimator performs pretty well. 

\begin{enumerate}
	\item Scenario 1 [Rectangular Signal]: The true signal $\theta^*$ is such that for every $(i_1, i_2) \in L_{2, n}$, we have
	$$\theta^*_{(i_1, i_2)} = 
	\begin{cases}
	1 \quad &\text{if} \;\; n/3 \leq i_1, i_2 \leq 2n/3\\
	0 \quad &\text{otherwise}
	\end{cases}.$$
	The corresponding plots are shown in Figure~\ref{fig:dc_box} when $n = 256$.
\begin{figure}[H]
	\centering
	\includegraphics[scale = 0.25]{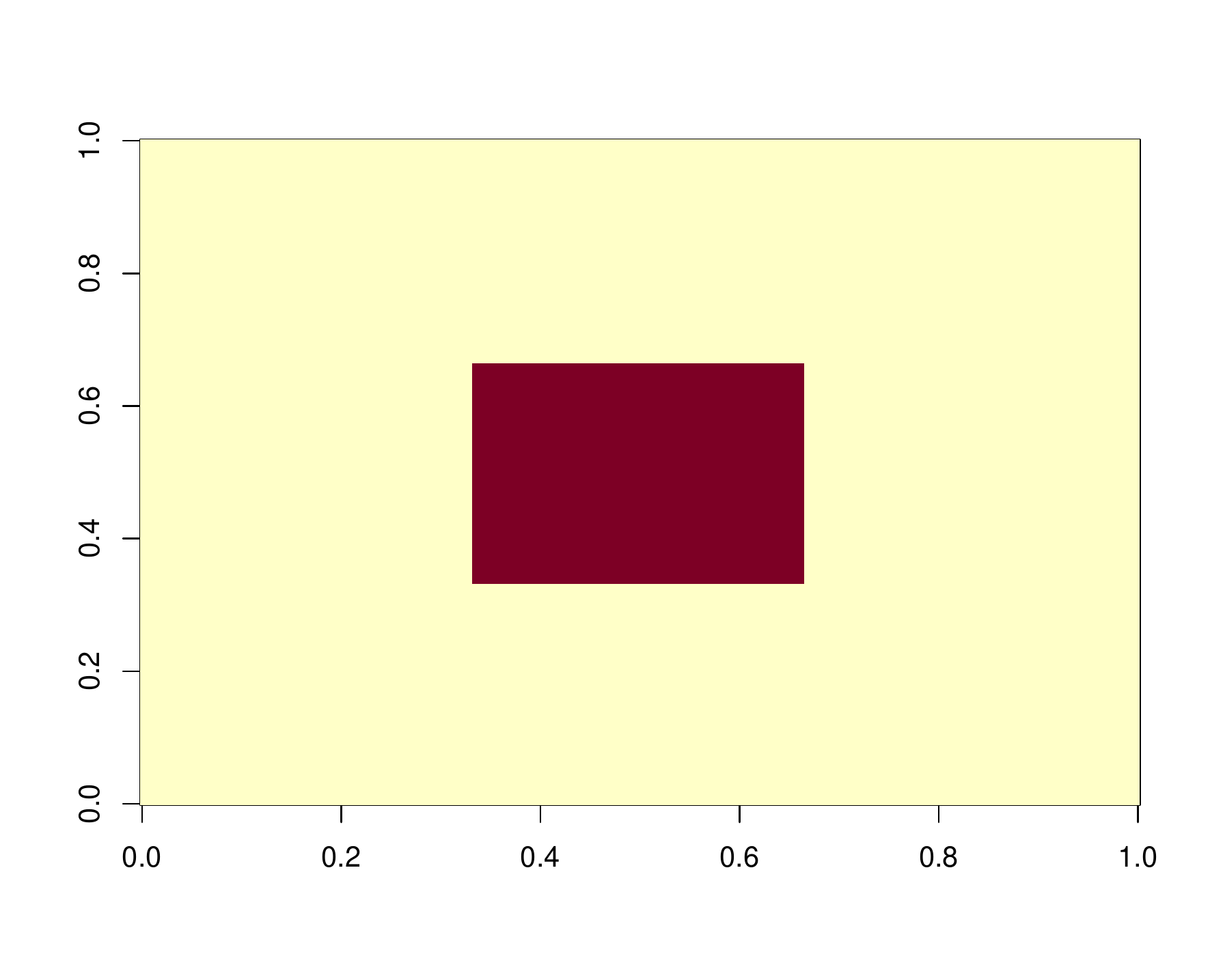} \includegraphics[scale = 0.25]{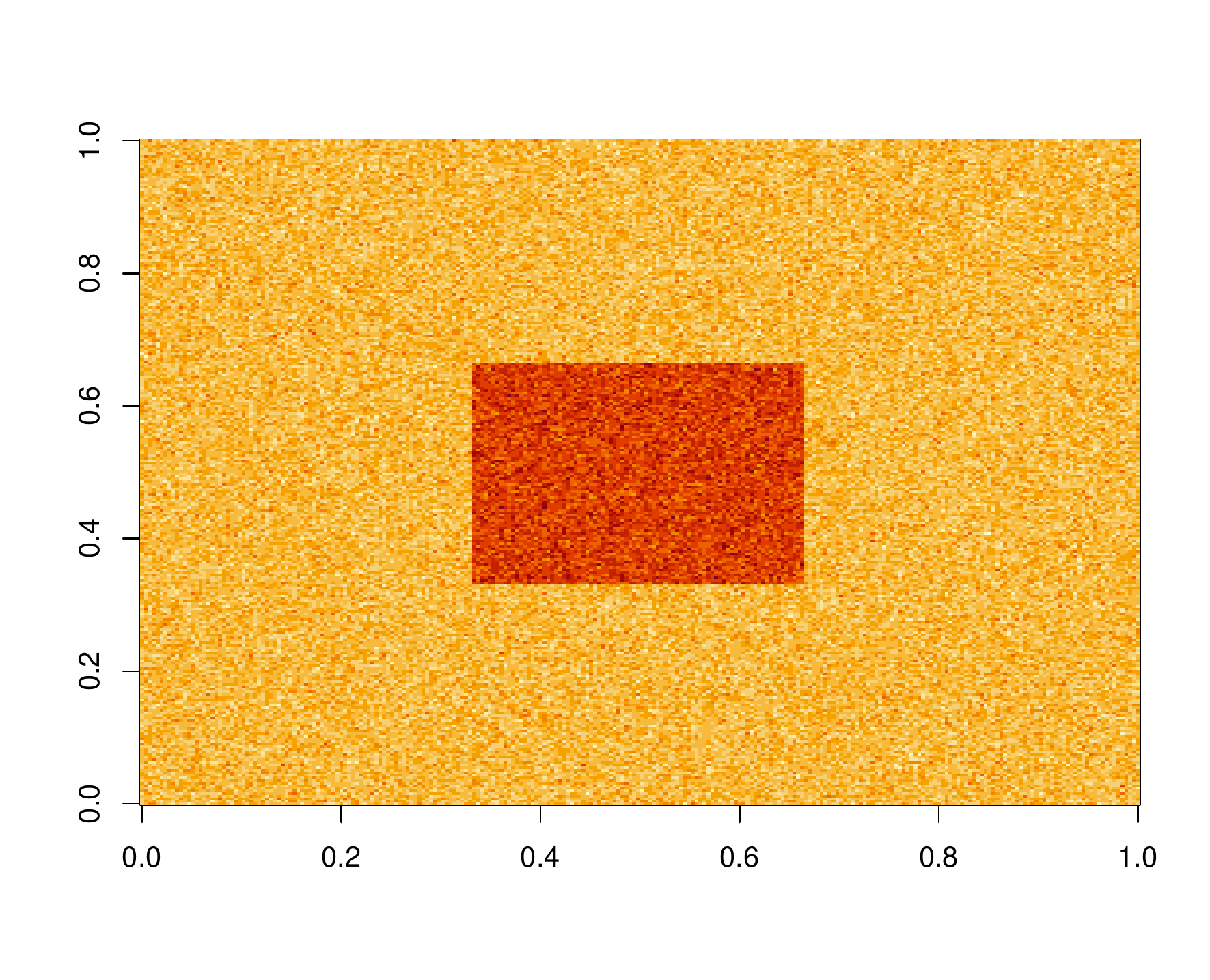} \includegraphics[scale = 0.25]{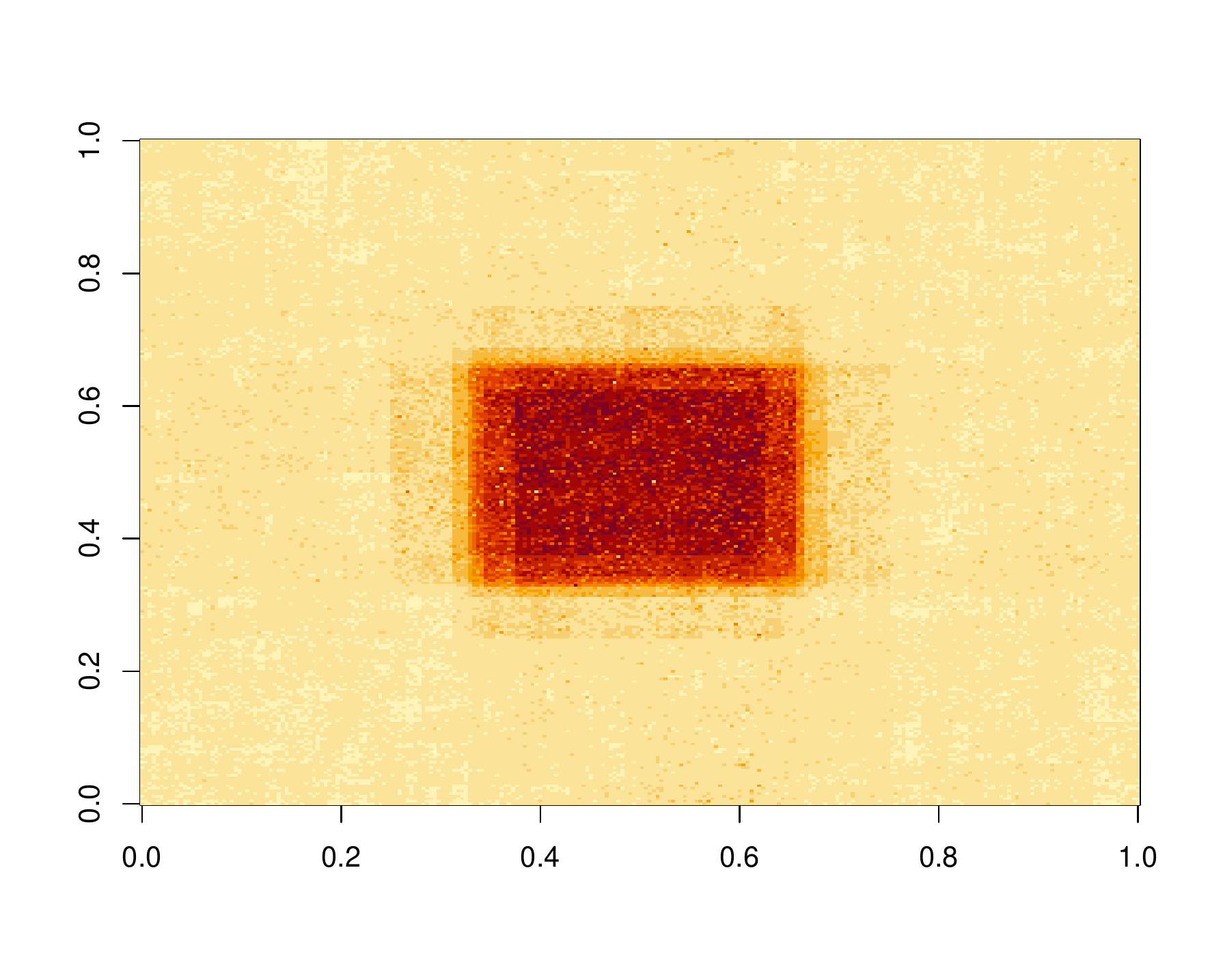}
	\caption{The first diagram refers to the true signal, the second one to the noisy signal and the third one to the estimated signal by the OMADRE estimator.}
	\label{fig:dc_box}
\end{figure}

\item Scenario 2 [Circular Signal]: The true signal $\theta^*$ is such that for every $(i_1, i_2) \in L_{2, n}$, we have
$$\theta^*_{(i_1, i_2)} = 
\begin{cases}
1 \quad &\text{if} \;\; \sqrt{(i_1 - n/2)^2 + (i_2 - n/2)^2} \leq n/4\\
0 \quad &\text{otherwise}
\end{cases}.$$
The corresponding plots are shown in Figure~\ref{fig:dc_circle} when $n = 256$.
\begin{figure}[H]
	\centering
	\includegraphics[scale = 0.25]{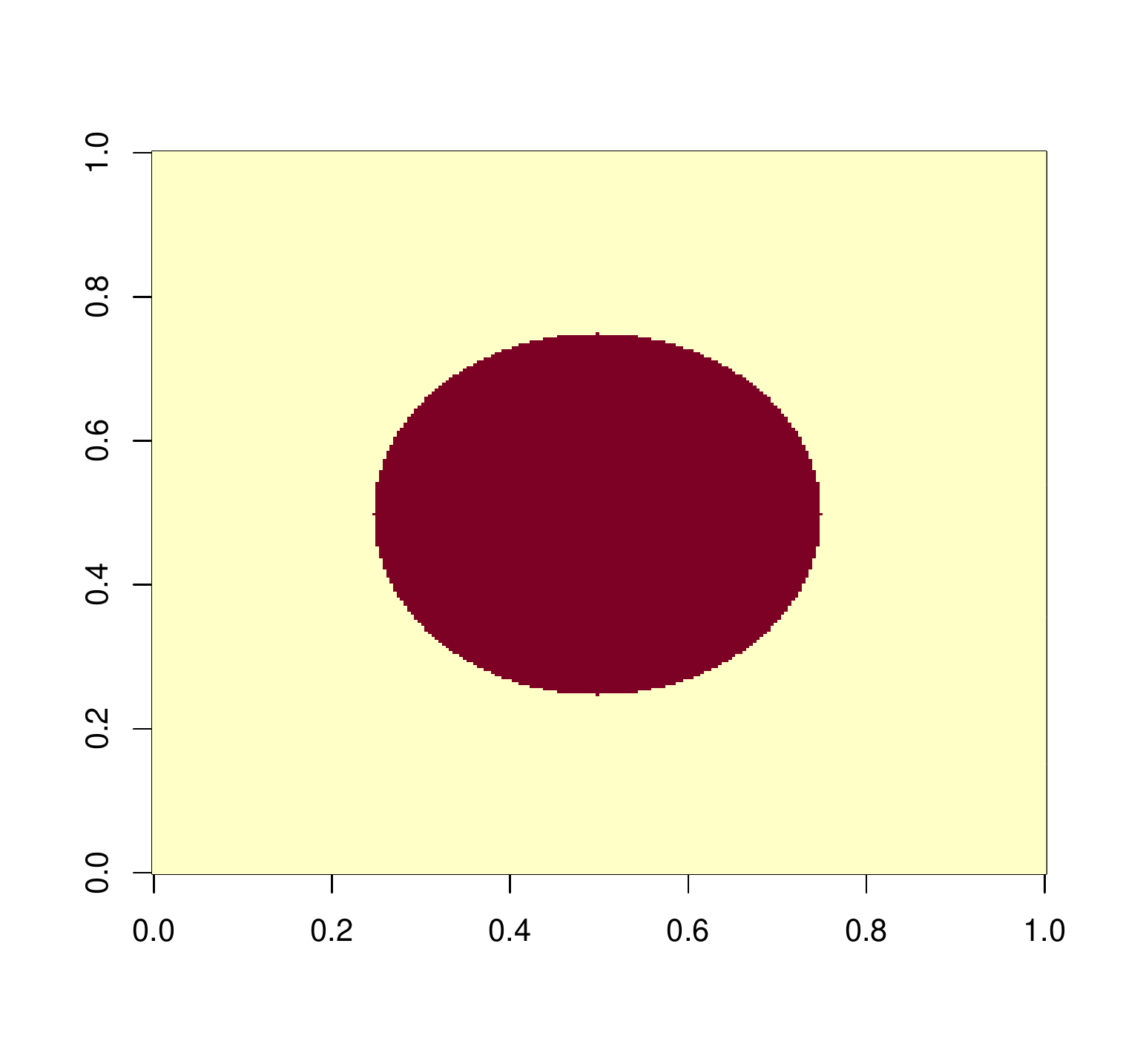} \includegraphics[scale = 0.25]{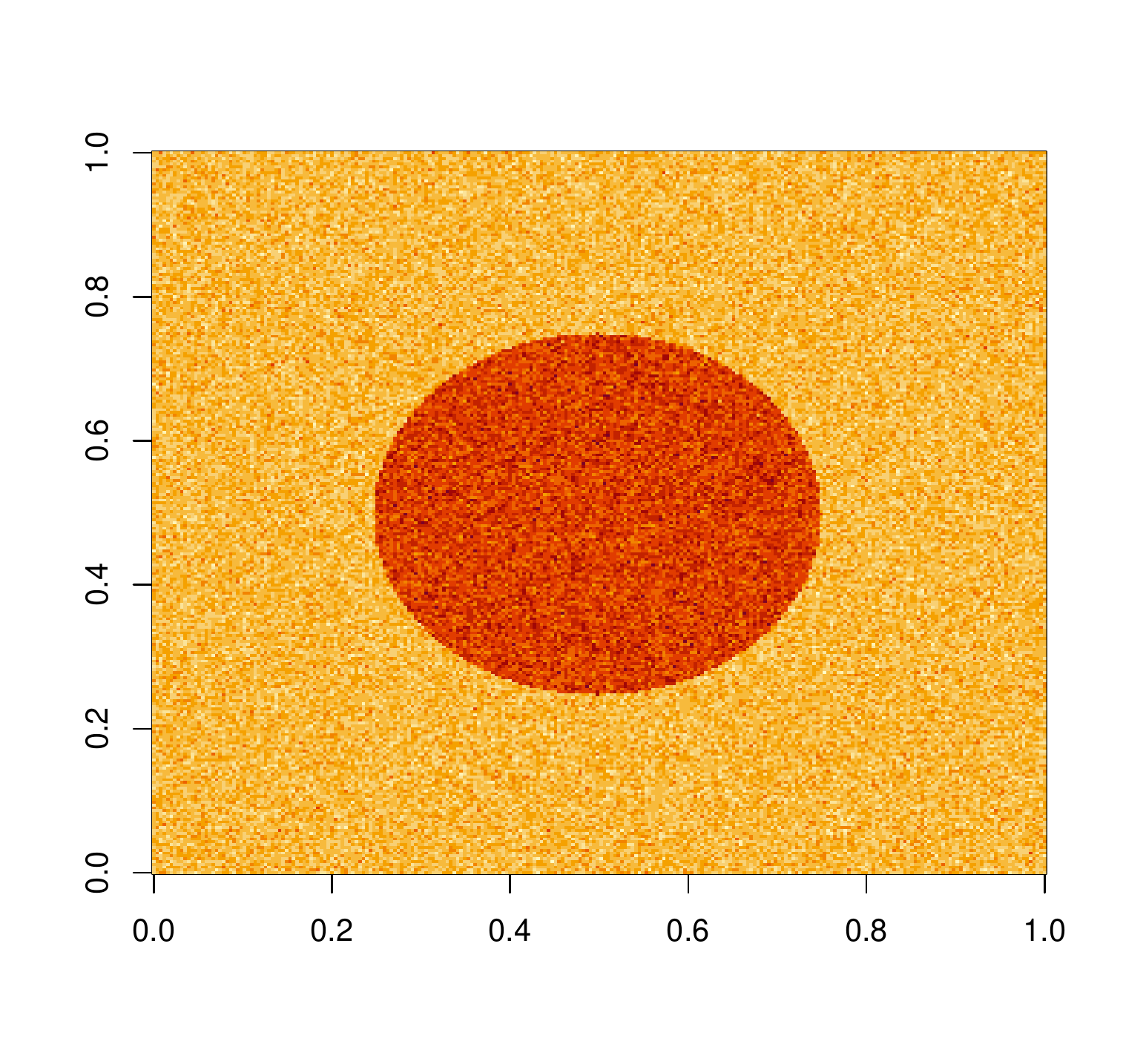} \includegraphics[scale = 0.25]{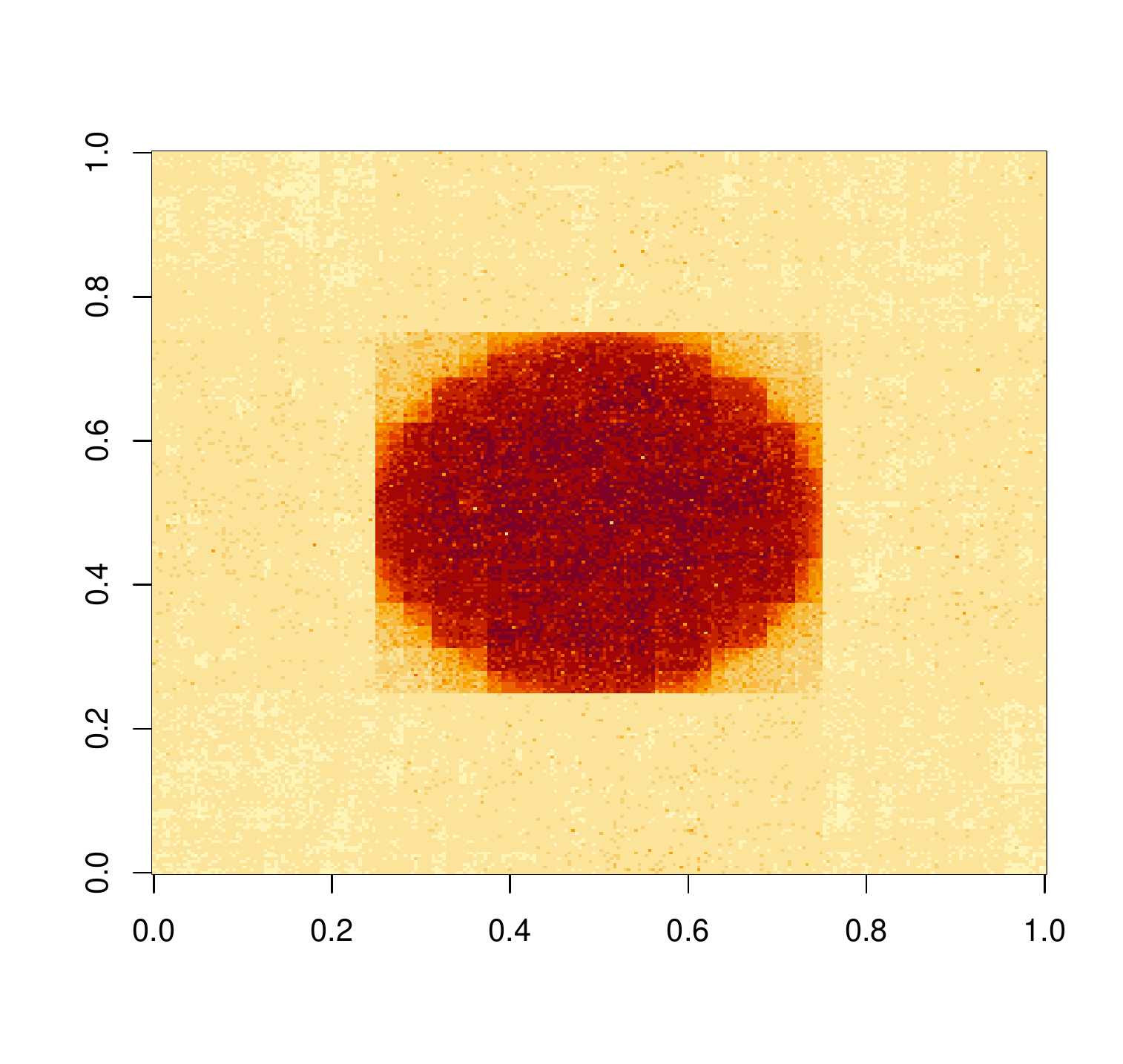}
	\caption{The first diagram refers to the true signal, the second one to the noisy signal and the third one to the estimated signal by the OMADRE estimator.}
	\label{fig:dc_circle}
\end{figure}

\item Scenario 3 [Sinusoidal Smooth Signal]: The true signal $\theta^*$ is such that for every $(i_1, i_2) \in L_{2, n}$, we have $\theta^*_{(i_1, i_2)} = f\left(i_1/n, i_2/n\right)$, where
$$f(x, y) = \sin(\pi x) \sin(\pi y).$$
The corresponding plots are shown in Figure~\ref{fig:dc_sinu} when $n = 256$.
\begin{figure}[H]
	\centering
	\includegraphics[scale = 0.25]{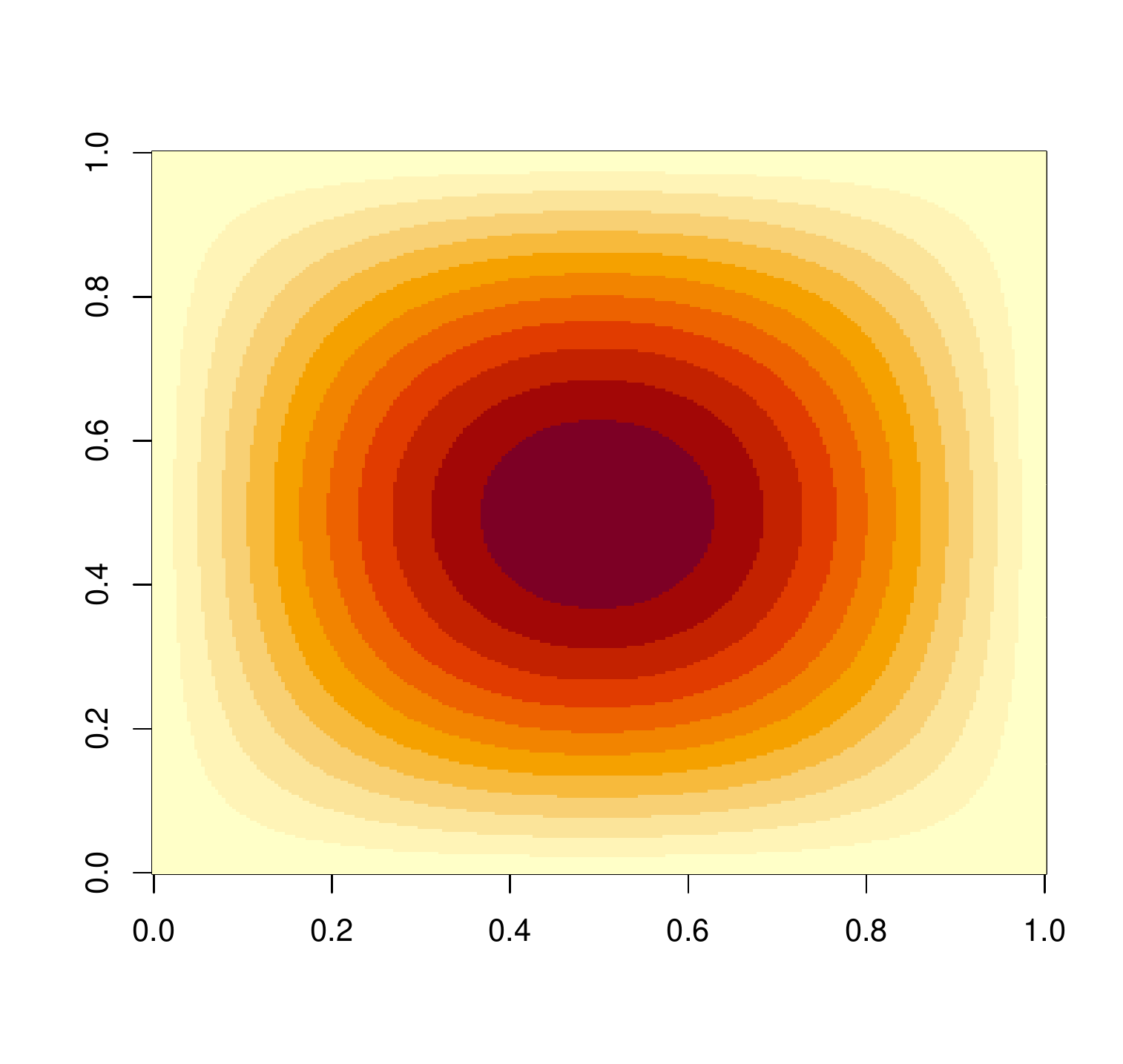} \includegraphics[scale = 0.25]{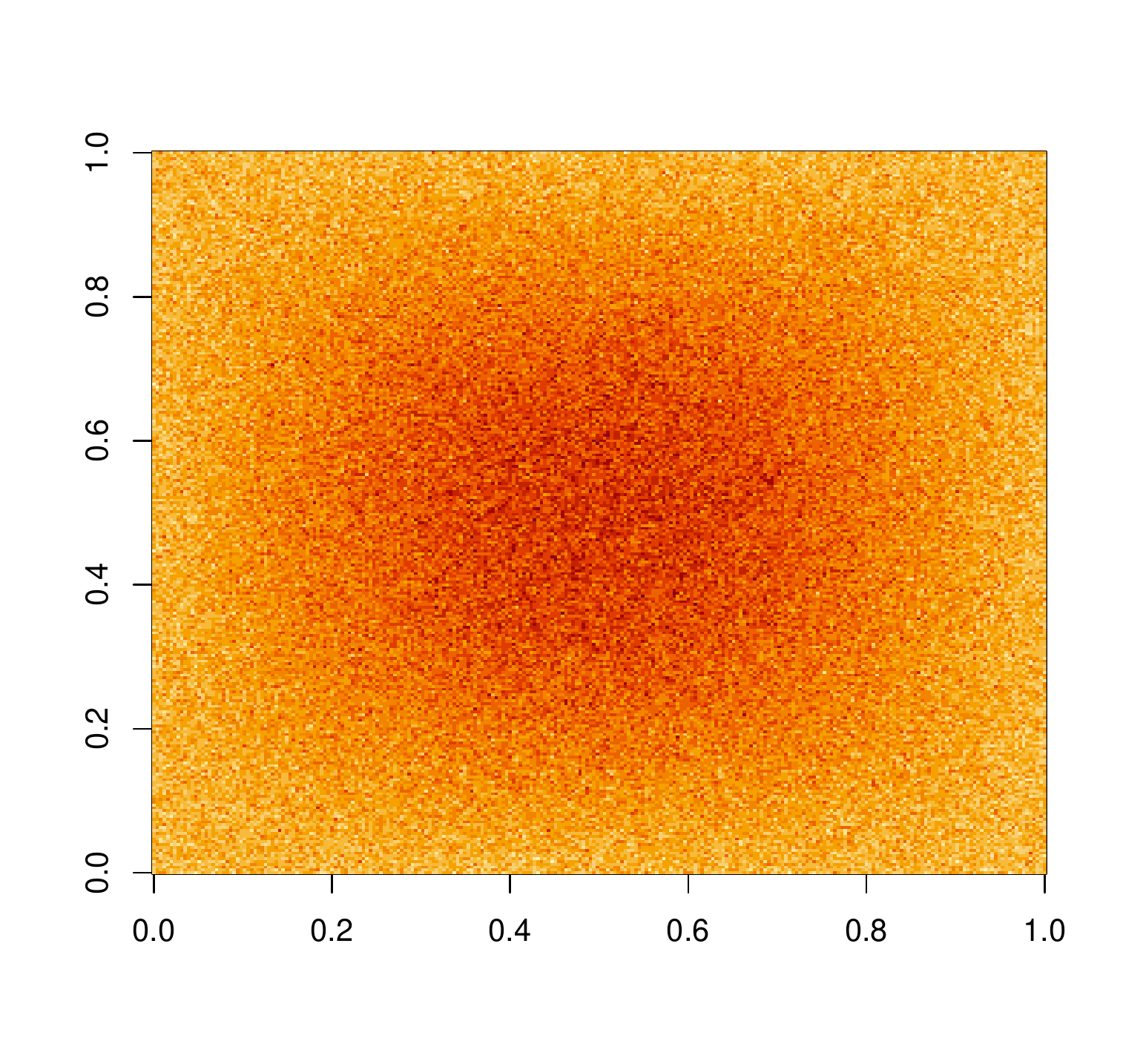} \includegraphics[scale = 0.25]{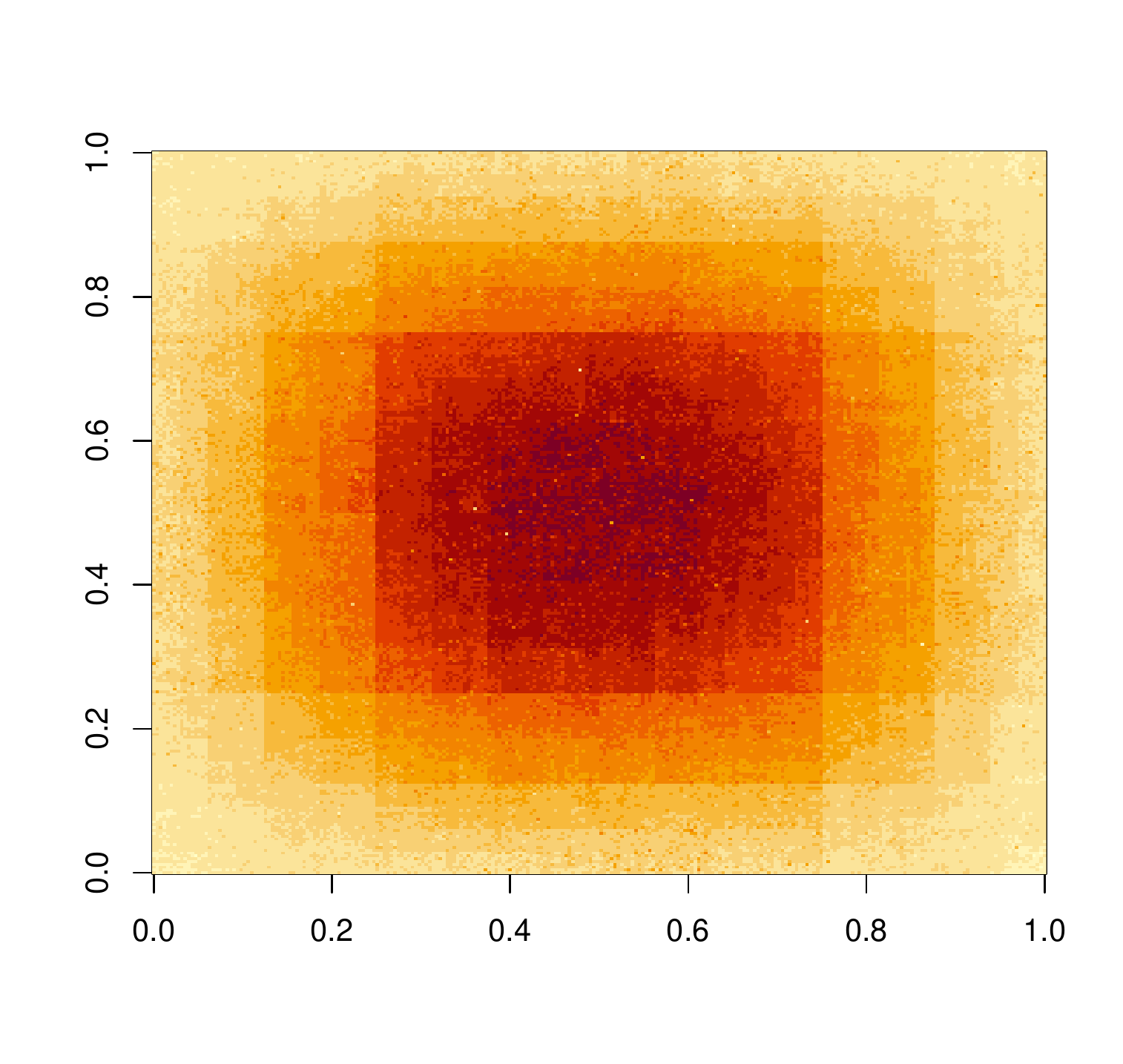}
	\caption{The first diagram refers to the true signal, the second one to the noisy signal and the third one to the estimated signal by the OMADRE estimator.}
	\label{fig:dc_sinu}
\end{figure}
\end{enumerate}

\begin{table}[H]
	\caption{MSEs of CV Dyadic CART estimator in different scenarios}
	\label{tab:omadre2d}
	\begin{tabular}{c|c|c|c}
		\hline
		$n \times n$ &Scenario 1  &Scenario 2 &Scenario 3\\
		\hline
		$64 \times 64$ & 0.035 & 0.037 & 0.014 \\
		$128 \times 128$ & 0.022 & 0.022 & 0.008 \\
		$256 \times 256$ & 0.012 & 0.013 & 0.005 \\
		\hline
	\end{tabular}
\end{table}

\section{Appendix}
\subsection{Proofs of Lemma~\ref{lem:complexity_online_mean} and Lemma~\ref{lem:complexity_online_regress}} We only prove 
Lemma~\ref{lem:complexity_online_regress} since it contains the proof of 
Lemma~\ref{lem:complexity_online_mean}. In the remainder of this subsection the constant $C$ always 
stands for an absolute constant whose precise value may change from one occurrence to the next. For 
every $s \in L_{d, n}$, we let $\mathcal S(s)$ denote the subcollection of all dyadic rectangles $S \subset 
L_{d, n}$ containing $s$.

At the outset of every round $t = 1 , \ldots, N$, we maintain several objects for every $S \in \mathcal S$. These 
include the weight $w_{S, t}$, the $L \times L$ matrix $X_{S, t} \coloneqq I + \sum_{s \in \rho[1:(t-1)] \cap S} 
x_s x_s^T$ where $L = |\mathcal F|$ and the vector $z_{S, t} = \sum_{s \in \rho[1:(t-1)] \cap S}\, y_s x_s \in 
\R^L$. We also store the indicator $I_{S, t} \in \{0, 1\}$ whether $S$ has had any datapoint upto round $t-1$ 
which is required to determine the set of active experts $A_t$ (recall step~2 of $\mathcal A$). In the beginning, $w_{S, 1} = \frac 1{|\mathcal S|}$ (recall the initialization step of $\mathcal A$), 
$X_{S, 1} = I$, $z_{S, 1} = 0$ and $I_{S, 1} = 0$ for all $S \in \mathcal S$. We first analyze the number of 
elementary operations necessary for computing the estimate $\hat y_{\rho(t)}$ and updating the matrices 
$(X_{S, t}; S \in \mathcal S(\rho(t)))$ as well as the indicators $I_{S, t}$ after the adversary reveals $\rho(t)$.

To this end observe that, we can visit all the rectangles in $A_t \subset \mathcal S(\rho(t))$ by performing 
binary search on each coordinate of $\rho(t) \in L_{d, n}$  in the lexicographic order and checking for the 
value of $I_{S, t}$. This implies, firstly, that $|\mathcal S(\rho(t))| \le (\log_2 2n)^d$ and secondly, that the 
number of operations required to update the indicators $I_{S, t}$'s is bounded by $(\log_2 2n)^d$.  Now let 
us recall from \eqref{def:online_regression} that, 
\begin{equation*}
\hat y_{\rho(t)}^{(S)} = X_{S+1, t}^{-1}\, z_{S, t} \cdot x_{\rho(t)},\:\: \mbox{where }X_{S, t+1} = X_{S, t} + x_{\rho(t)} x_{\rho(t)}^T.
\end{equation*}
Computing $X_{S+1, t}$ and its inverse, and the subsequent multiplication with $z_{S, t}$ require at most $C 
SL^3$ and $CL^2$ many basic operations respectively. Evaluating the inner product with $x_s$ afterwards  
take at most $C L$ many basic steps. 
Thus, we incur $C L^3$ as the total cost for computing 
$w_{S, t}T_{\lambda}(\hat y_{\rho(t)}^{(S)})$ and updating $X_{S, t}$ for each $S \in \mathcal S(\rho(t))$. 
Calculating $\hat y_{\rho(t)}$ from the numbers $w_{S, t}T_{\lambda}(\hat y_{\rho(t)}^{(S)})$'s (see step~3 
of $\mathcal A$), where $S \in \mathcal S(\rho(t))$, requires $C |\mathcal S(\rho(t))|$ many additional steps. 
Therefore, the combined cost for computing $\hat y_{\rho(t)}$ and updating $X_{S, t}$'s for all $S \in 
\mathcal S(\rho(t))$ is bounded by $C L^3 |\mathcal S(\rho(t))| =  C L^3(\log_2 2n)^d$.

After the adversary reveals $y_{\rho(t)}$, we need to update the weights $w_{S, t}$, the vectors $z_{S, 
t}$ and the indicators $I_{S, t}$ for all $S \in \mathcal S(t)$ (see step~4 of $\mathcal A$). For this we first need to compute the numbers $w_{S, 
t}\e^{-\alpha \ell_{S, t}}$ for all $S \in A_t$ and this takes $C |\mathcal S(\rho(t))| = C (\log_2 
2n)^d$ many basic operations. It takes an additional $C (\log_2  2n)^d$ many basic operations in order to 
compute the sums $\sum_{S \in A_t} w_{S, t}$ and $\sum_{S \in \mathcal A_t} w_{S, t} 
\e^{-\alpha \ell_{S, t}}$. Using these numbers, we can now update the weights as
\begin{equation*}
w_{S, t + 1} = \frac{w_{S, t} \e^{-\alpha \ell_{S, t}}}{\sum_{S \in A_t} w_{S, t} \e^{-\alpha \ell_{S, t}}} \sum_{S \in A_t} w_{S, t} 
\end{equation*}
and this also involves $C (\log_2 n)^d$ many elementary operations. Updating the vector $z_{S, t}$ to 
$z_{S, t + 1} = z_{S, t} + y_{\rho(t)} x_{\rho(t)}$ takes at most $C L$ many basic steps for every $S$ and 
hence $C L (\log_2 n)^d$ many steps in total.

Putting everything together, we get that the computational complexity of OLRADRE  is bounded by $C L^3 N (\log _2 n)^d$.\qed

\subsection{Proof of Theorem~\ref{thm:pcconst}}
Recall the definition of $\overline{R}(\theta,P)$ for any partition $P$ of $K = L_{d,n}$ and a 
$\theta \in \Theta_{P}$ given right after~\eqref{eq:regret_bnd}. It turns out that for the 
online averaging rule, one can give a clean bound on $\overline{R}(\theta,P)$ which is 
stated next as a proposition.

\begin{proposition}\label{prop:ruleprop1}
	Let $y_t = \theta^*_t + \sigma \ep_t$ for $t \in K = L_{d, n}$ where $\sigma > 0$ and $\ep_t$'s are 
	independent, mean zero sub-Gaussian variables with unit dispersion factor. Then we have for 
	any partition $P \in \mathcal{P}_{\mathsf T}$, where $\mathsf T \subset K $, and any $\theta \in 
	\Theta_{P}$,
	\begin{equation}\label{eq:expec_online_mean_regret1}
	\overline {\mathcal R}(\theta,P) \leq
	C |P| \, (\|\theta_{\mathsf T}^*\|_{\infty}^2 + \sigma^2 \log \e N)  \log \e N.
	\end{equation}
	where $C > 1$ is some absolute constant.
\end{proposition}

\begin{proof}
	We have 
	\begin{align*}
		\overline {\mathcal R}(\theta,P) = \E \:\mathcal{R}(y, \theta, P) \leq |P|\:\: \E \sup_{\rho, S \in P} \Big(\sum_{t: \rho(t) \in S}(y_{\rho(t)} - \hat{y}_{\rho(t)}^{(S)})^2  - \|y_S - \theta_{S} \|^2\Big).
	\end{align*}
	
	Now, the following deterministic lemma is going to be of use to us. 
	\begin{lemma}
		\label{lem:online_mean_regret1}
		Let $z_1, \ldots, z_T$ be an arbitrary sequence of numbers and $\hat z_t \coloneqq 
		\frac{1}{t-1}\sum_{s = 1}^{t-1} z_s$ for $t = 2, \ldots, T$ where $\hat z_1 = 0$. Then, we have
		\begin{equation}\label{eq:online_mean_regret1}
		\|z - \hat z\|^2 - \|z - \bar z\|^2  \le 4 \|z\|_{\infty}^2 \log \e T.
		\end{equation}
	\end{lemma}

For a proof of the above lemma, see, e.g., Theorem~1.2 in \cite{orabona2019modern}. Using the above 
deterministic lemma and the previous display, we can write for any partition $P \in \mathcal{P}_{\mathsf T}$ 
and $\theta \in \Theta_{P}$,
	\begin{align}\label{eq:online_mean_regret_bnd}
		\overline {\mathcal R}(\theta,P) & \leq 4 |P|\, \,  \E \, \sup_{\rho,S \in P} \Big(\sum_{t: \rho(t) \in S}(y_{\rho(t)} - \hat{y}_{\rho(t)}^{(S)})^2  - \|y_S - \overline{y_{S}} \|^2\Big) \nonumber 
		\\ & \leq  4 |P| \:\:\E \|y\|_{\infty}^2 \:\log \e n \leq C |P|\, (\|\theta^*\|_{\infty}^2 + \sigma^2 \log \e N)  \log \e N.
	\end{align}
	where $\overline{y_{S}}$ denotes the mean of the entries of $y_{S}$ and we deduce the last inequality 
	from a standard upper bound on the tail of sub-Gaussian random variables. 
\end{proof}

The following corollary is a direct implication of Theorem~\ref{thm:generic_online} and 
Proposition~\ref{prop:ruleprop1} applied to the particular setting described at the beginning of 
Section~\ref{sec:omadre}.
\begin{corollary}\label{cor:meanresult}
	Let $\mathsf T$ be any subset of $K.$ Let $\hat{\theta}^{OM}$ denote the OMADRE predictor. 
	There exists an absolute constant $C > 1$ such that for $\lambda \geq C (\sigma \sqrt{\log N} 
	\vee\|\theta^*\|_{\infty})$, one has for any non-anticipating ordering $\rho$ of $K$,
	\begin{align}\label{eq:generic_online1}
		\E \|\hat \theta^{OM}_{\mathsf T} - \theta_{\mathsf T}^* \|^2  \le 
		\inf_{\substack{P \in {\mathcal 
					P}_{\dpt,T}\\ \theta \in \Theta_P \subset \R^{T}}} &\big( \|\theta_{\mathsf T}^* - \theta \|^2 + C \lambda^2 |P| \log 2^d N\big) + \frac{\sigma^2 + \lambda^2}{|\mathsf T|}.
	\end{align}
\end{corollary}

We are now ready to prove Theorem~\ref{thm:pcconst}.

\begin{proof}
	The proof directly follows from~\eqref{eq:generic_online2} and Proposition~\ref{prop:ruleprop1}.
\end{proof}

\begin{proof}[Proof of Theorem~\ref{thm:pcconst}]
	Fix any partition $P \in \mathcal{P}_{\all, \mathsf T}$ and any $\theta \in \Theta_{\mathsf P}.$ Consider a dyadic refinement of $P$ which we denote by $\tilde{P}$. By definition, 
	$\tilde{P} \in \mathcal{P}_{\dpt, \mathsf T}$ and  $\theta \in \Theta_{\tilde{P}}.$ Therefore, we can use the bound 
	in~\eqref{eq:generic_online1} given in Corollary~\ref{cor:meanresult}. The proof is then finished by noting 
	that $|\tilde{P}| \leq |P| (\log \e n)^d.$
\end{proof}

\subsection{Proof of Theorem~\ref{thm:TV_slow_rate}}
It has been shown in~\cite{chatterjee2019adaptive} that the class of functions $\mathcal{BV}_{d, n}(V^*)$ is well-approximable by 
piecewise constant functions with dyadic rectangular level sets which makes it natural to study the OMADRE estimator for this function class. 


The following result was proved in \cite{chatterjee2019adaptive} (see Proposition~8.5 in the arxiv version).
\begin{proposition}\label{prop:division}
	Let $\theta \in \R^{L_{d,n}}$ and $\delta > 0.$ Then there exists a dyadic partition $P_{\theta,\delta} = (R_1,\dots,R_k)$ in $\mathcal P_{\rdp}$ such that
	\newline
	a) $k = |P_{\theta,\delta}| \leq 1 +  \log_2 N \:\:\big(1 + \frac{\TV(\theta)}{\delta}\big)$, and for all $i \in [k]$,\newline
	c) $\TV(\theta_{R_i}) \leq \delta$, and \newline
	d) $\mathcal{A}(R_i) \leq 2$ \newline
	where $\mathcal{A}(R)$ denotes the aspect ratio of a generic rectangle $R.$ 
\end{proposition}
Let $\Pi_{\Theta_{P_{\theta, \delta}}} \coloneqq \Pi_{P_{\theta, \delta}}$ denote the orthogonal projector onto the subspace $\Theta_{P_{\theta, \delta}}$ of 
$\R^{L_{d, n}}$ comprising functions that are constant on each $R_i \in P_{\theta, \delta}$. 
It is clear that $\Pi_{P_{\theta, \delta}} \theta(a) = \overline \theta_{R_i}$ --- 
the average value of $\theta$ over $R_i$ --- for all $a \in R_i$ and $i \in [k]$. We will use 
$\Pi_{P_{\theta^*, \delta}} \theta^*$ as $\theta$ in our application of \eqref{eq:generic_online1} in this case. 
In order to estimate $\|\theta^* - \Pi_{\Theta_{P_{\theta^*, \delta}}}\theta^*\|$, we would need the following 
approximation theoretic result.
\begin{proposition}
	\label{prop:gagliardo}
	Let $\theta \in \R^{\otimes_{i \in [d]}[n_i]}$ and 
	$$\overline \theta \coloneqq \sum_{(j_1, j_2, \ldots, j_d) \,\in\, \otimes_{i\in [d]}[n_i]}\theta[j_1, j_2, \ldots, j_d] / \prod_{i \in [d]}n_i$$ be the average of 
	the elements of $\theta$. Then for every $d > 1$ we have,
	\begin{equation}\label{eq:var_bnd_d>1}
	\sum_{(j_1, j_2, \ldots, j_d) \,\in\, \otimes_{i\in [d]}[n_i]}|\theta[j_1, j_2, \ldots, j_d] - \overline \theta|^{2} \leq \Big(1 + \:\max_{i, j \in [d]}\frac{n_i}{n_j}\Big)^{2}\TV(\theta)^{2}\,.
	\end{equation}
	For $d = 1$, on the other hand, we have
	\begin{equation}\label{eq:var_bnd_d=1}
	\sum_{j \in [N]}|\theta[j] - \overline \theta|^{2} \leq N\,\TV(\theta)^{2}\,.
	\end{equation}
\end{proposition}
See \cite{chatterjee2019adaptive} (Proposition $8.7$ in the arxiv version) for a proof of \eqref{eq:var_bnd_d>1} and~\cite{chatterjee2019new} (Lemma $10.3$ in the arxiv version) for 
\eqref{eq:var_bnd_d=1}. Propositions~\ref{prop:division} and \ref{prop:gagliardo} together 
with the description of $\Pi_{P_{\theta, \delta}}$ as the operator that projects $\theta$ 
onto its average value on each rectangle $R_i$, imply that
\begin{equation}\label{eq:piecewise_const_TVd>1}
\|\theta - \Pi_{P_{\theta, \delta}}\theta\|^2 \le C |P_{\theta, \delta}| \delta^2 = C \log_2 2N \, (\delta^2 + \delta \TV(\theta) )
\end{equation}
for $d > 1$ whereas for $d = 1$,
\begin{equation}\label{eq:piecewise_const_TVd=1}
\|\theta - \Pi_{P_{\theta, \delta}}\theta\|^2 \le N \delta^2. 
\end{equation}

\begin{proof}[Proof of Theorem~\ref{thm:TV_slow_rate}]
	We get by plugging the bounds from \eqref{eq:piecewise_const_TVd>1} and item~(a) in 
	Proposition~\ref{prop:division} --- both evaluated at $\theta = \Pi_{P_{\theta_{\mathsf 
	T}^*,\delta}} \theta_{\mathsf T}^*$ --- into Corollary~\ref{cor:meanresult} 
	\begin{align*}
		\E \|\hat \theta^{OM}_{\mathsf T} - \theta_{\mathsf T}^* \|^2 \le \, \inf_{\delta > 0} \, C&\big((\delta^2 + 
		\delta V^*_{\mathsf T}) \log 2^d N + \lambda^2 (\log 2^d N)^2 \big(1 + \frac{V_{\mathsf T}^*}{\delta}\big) \big) + \frac{\sigma^2 + \lambda^2}{|\mathsf T|}.
	\end{align*}
	Now putting $\delta = \lambda (\log 2^d N)^{1/2}$ in the above display, we obtain \eqref{eq:TV_slow_rate_d>1}. 
	
	For $d = 1$, we follow the exact same steps except that we now use the bound \eqref{eq:piecewise_const_TVd=1} in lieu of \eqref{eq:piecewise_const_TVd>1} to deduce
	\begin{align*}
		\E \|\hat \theta^{OM}_{\mathsf T} - \theta_{\mathsf T}^* \|^2 \le \, \inf_{\delta > 0} \, C&\big( |\mathsf T| \delta^2 + \lambda^2 (\log 2^d N)^2 \big(1 + \frac{V_{\mathsf T}^*}{\delta}\big) \big) + \frac{\sigma^2 + \lambda^2}{|\mathsf T|}.
	\end{align*}
	This immediately leads to \eqref{eq:TV_slow_rate_d=1} upon setting $\delta = (V_{\mathsf T}^*)^{1/3} 
	\lambda^{2/3} (\log 2^d N)^{2/3} |\mathsf T|^{-1/3}$.
\end{proof}

\subsection{Proof of Theorem~\ref{thm:pcpoly}}
We take a similar approach as in the proof of Theorem~\ref{thm:pcconst}. Let us begin with an upper bound 
on the regret of the estimator (see, e.g., \cite[pp.~38--40]{rakhlin2012statistical} for a proof). \begin{proposition}[Regret bound for 
	Vovk-Azoury-Warmuth forecaster]\label{prop:online_regress_regret}
	Let $(z_1, x_1), \ldots,$ $(z_T, x_T) \in \R \times \R^d$ and define for $t = 1, \ldots, T$ (cf.~\eqref{def:online_regression}),
	\begin{equation*}
		\hat z_t \coloneqq \hat \beta_s \cdot x_s \mbox{ where } \hat \beta_s = \Big(I + \sum_{s = 1}^{t-1} x_s x_s^T \Big)^{-1} \big(\sum_{s = 1}^{t-1} y_s x_s \big).
	\end{equation*} Then, we have
	\begin{equation}\label{eq:online_regr_regret}
		\sum_{t \in [T]} (\hat z_t  - z_t)^2 - \inf_{\beta \in \R^d} \big\{\sum_{t \in [T]}(z_t - 
		\beta \cdot x_t)^2 + \|\beta\|^2  \big\} \le d \|z_\infty\|^2 \log (1 +  T \max_{t \in 
			[T]}\|x_t\|^2/d).
	\end{equation}
\end{proposition}

Using similar arguments as in \eqref{eq:online_mean_regret_bnd}, but applying 
\eqref{eq:online_regr_regret} 
instead of \eqref{eq:online_mean_regret1} for bounding the regret, we get for any $P \in \mathcal P_{\mathsf 
T}$ and $\theta \in \Theta_P$,
\begin{equation*}
\overline{R}(\theta, P)	  \le C_{m, d} |P|\, (s_{m, \infty}(\theta)^2 + \sigma^2 \log \e N)  \log \e N
\end{equation*}
where $C_{m, d} > 1$ depends only on $m$ and $d$. The remaining part of the proof is similar to that of 
Theorem~\ref{thm:pcconst}.

\subsection{Proof of Theorem~\ref{thm:trendfilter_slow_rate}}
The proof requires, first of all, that  the class $\mathcal{BV}^{(m)}_{n}(V^*)$ is well-approximable by 
piecewise polynomial functions with degree at most $m-1$. To this end we present the following result 
which was proved in \cite{chatterjee2019adaptive} (see Proposition~8.9 in the arxiv version).
\begin{proposition}\label{prop:piecewise}
	Fix a positive integer $m > 1$ and $\theta \in \R^n$, and let $V^{m}(\theta) \coloneqq V$. For 
	any $\delta > 0$, there exists a partition $P_{\theta, m, \delta}$ (of $L_{1, n}$) in $\mathcal P_{\dpt}$ and 
	$\theta' \in \Theta_{P_{\theta, m,\delta}}(\{F_S: S \in \mathcal S\})$ such that \newline
	a) $|P_{\theta, m, \delta}| \leq C\, \delta^{-1/m}$ for an absolute constant $C$, 
	\newline
	b) $\|\theta - \theta^{'}\|_{\infty} \leq V \delta$, \mbox{ and} \newline
	c) $\max_{S \in P_{\theta, m, \delta},\, 0 \le j < m} \, n^{j} \beta_{j, S} \le C_m \max_{0 \le j < m} 
	n^{j}\|D^{j}(\theta)\|_{\infty}$ where $\theta' \equiv \beta_{0, S} + \ldots + \beta_{m-1, S} x^{m-1}$ on $S$ 
	and $C_m$ is a constant depending only on $m$.
\end{proposition} 
Proposition~\ref{prop:piecewise} immediately gives us
\begin{equation}\label{eq:piecewise_poly_highTVd=1}
	\|\theta - \Pi_{\Theta_{P_{\theta, m, \delta}}}\theta\|^2 \le N V^2 \delta^2.
\end{equation}
(recall that $N = n$ since $d = 1$). 

\begin{proof}[Proof of Theorem~\ref{thm:trendfilter_slow_rate}]
	Given any $\delta > 0$, let $\theta'_\delta$ denote the vector given by 
	Proposition~\ref{prop:piecewise} for $\theta = \theta^*_{\mathsf T}$. Then from 
	Proposition~\ref{prop:online_regress_regret} and the item~c) in 
	Proposition~\ref{prop:piecewise}, we get, for some constant $C_m$ depending only on $m$,
\begin{align}\label{eq:expected_regret_bnd_trendfilter}
	\overline {\mathcal R}(\theta'_{\delta}, P_{\theta, m, \delta}) &\le C_m|P_{\theta, m, \delta}|\,(\, \|\theta^*_{\mathsf T}\|_{m-1, \infty}^2 + \E \|y_{\infty}\|^2 \log \e N \,) \nonumber \\ 
	&\le C_m|P_{\theta, m, \delta}|\, (\|\theta^*_{\mathsf T}\|_{m-1, \infty}^2 + \sigma^2 \log \e N)  \log \e N
\end{align}
	where the last step follows from a similar computation as in \eqref{eq:online_mean_regret_bnd} (observe 
	that $\|\theta^*_{\mathsf T}\|_{m-1, \infty} \ge \|\theta^*_{\mathsf T}\|_{\infty}$). Now, since $\log 
	|\mathcal S| = \log 2 N$ (see around \eqref{eq:no_of_expert}), we get by 
	plugging the bounds from \eqref{eq:piecewise_poly_highTVd=1}, item~(a) in 
	Proposition~\ref{prop:piecewise} and \eqref{eq:expected_regret_bnd_trendfilter} into 
	\eqref{eq:generic_online2} that for any $\lambda  \ge C (\sigma \sqrt{\log N} \vee\|\theta^*_{\mathsf 
		T}\|_{\infty})$,
	\begin{align*}
		\E \|\hat \theta^{OL}_{\mathsf T} - \theta_{\mathsf T}^* \|^2  \le \, \inf_{\delta > 0} \, C&\big( |\mathsf T| (V_{\mathsf T}^*)^2 \delta^2 +  C_m \delta^{-1/m}\, (\lambda + \|\theta^*_{\mathsf T}\|_{m-1, \infty})^2\log \e N\big) + \frac{\sigma^2 + \lambda^2}{|\mathsf T|}
	\end{align*} where $C_m > 0$ depends only on $m$. Now putting $$\delta = C_m^{\frac{m}{2m + 
			1}}(\lambda + \|\theta^*_{\mathsf T}\|_{m-1, \infty})^{\frac{2m}{2m + 1}}(V_{\mathsf T}^*)^{-\frac{2m}{2m + 1}} (\log \e N)^{\frac{m}{2m + 1}} |\mathsf T|^{-\frac{m}{2m+1}}$$ in the above display, we obtain \eqref{eq:trendfilter_slow_rate_d=1}.
\end{proof}

\bibliographystyle{chicago}
\def\noopsort#1{}

\end{document}